\documentclass[12pt]{amsart}
\usepackage{amsfonts, amsmath, amssymb, amsthm}
\usepackage{changebar}
\usepackage{xypic}

\newtheorem{lemma}{Lemma}
\newtheorem{theorem}{Theorem}

\newtheorem{proposition}{Proposition}

\newtheorem{definition}{Definition}

\newtheorem{remark}{Remark}
\newtheorem{fact}{Fact}

\newtheorem{assum}{Assumption}

\newcommand{\nth}[1]{$#1 {\rm - th }$}

\newcommand{\rad}{{\rm rad } }

\newcommand{\sexp}{{\rm sexp} }
\newcommand{\ord}{{\rm ord} }

\newcommand{\prk}{p{\rm -rk } }

\newcommand{\zprk}{\Z_p\hbox{-rk}}

\newcommand{\Z}{\mathbb{Z}}

\newcommand{\rg}[1]{\mbox{\bf #1}}
\newcommand{\eu}[1]{\mathfrak{#1}}
\newcommand{\id}[1]{\mathcal{#1}}
\newcommand{\Gal}{\mbox{ Gal }}
\newcommand{\Gals}{\hbox{\tiny Gal}}
\newcommand{\Ker}{\mbox{ Ker }}

\newcommand{\rest}{\mbox{Rest}}

\newcommand{\rf}[1]{(\ref{#1})}

\newcommand{\Norm}{\mbox{\bf N}}

\newcommand{\F}{\mathbb{F}}

\newcommand{\U}{\mathbb{U}}
\newcommand{\V}{\mathbb{V}}
\newcommand{\B}{\mathbb{B}}
\newcommand{\K}{\mathbb{K}}

\newcommand{\LK}{\mathbb{L}}
\newcommand{\KH}{\mathbb{H}}

\newcommand{\KL}{\mathbb{L}}

\newcommand{\M}{\mathbb{M}}
\newcommand{\T}{\mathbb{T}}
\newcommand{\D}{\mathbb{D}}
\newcommand{\Q}{\mathbb{Q}}

\newcommand{\N}{\mathbb{N}}

\newcommand{\pinf}{^{1/p^{\infty}}}

\def\ra{\rightarrow}

\def\wh{\widehat}

\newcommand{\ran}{\rangle}
\newcommand{\lan}{\langle}

\begin{document}
{\obeylines \small
\vspace*{-1.0cm}
\hspace*{4.5cm}Io penso positivo perch\'{e} son vivo
\hspace*{4.5cm}E finch\'e son vivo
\hspace*{4.5cm}Niente e nessuno al mondo
\hspace*{4.5cm}Potr\'a fermarmi dal ragionare\footnote{Giovanotti: {\em Io penso positivo}.}.
\vspace*{0.5cm}
\hspace*{6.8cm}To Ralph Greenberg

\vspace*{0.5cm}
\smallskip} 
\title[The Gross Conjecture] {The Gross-Kuz'min Conjecture for CM
fields } 
\author{Preda Mih\u{a}ilescu} \address[P. Mih\u{a}ilescu]{Mathematisches
Institut der Universit\"at G\"ottingen}
\email[P. Mih\u{a}ilescu]{preda@uni-math.gwdg.de}

\date{Version 4.0 \today}
\vspace{0.5cm}
\begin{abstract}
  Let $A' = \varprojlim_n A'_n$ be the projective limit of the
  $p$-parts of the ideal class groups of the $p$ integers in the
  $\Z_p$-cyclotomic extension $\K_{\infty}/\K$ of a CM number field
  $\K$. We prove in this paper that the $T$-part $(A')^-(T) = \{ 1 \}$
  for CM extensions $\K/\Q$. This fact has been conjectured for
  arbitrary fields $\K$ by Kuz'min in 1972 and was proved by Greenberg
  in 1973, for abelian extensions $\K/\Q$. Federer and Gross had shown
  in 1981 that $(A')^-(T) = \{ 1 \}$ is equivalent to the
  non-vanishing of the $p$-adic regulator of the $p$-units of $\K$.
\end{abstract}
\maketitle
\tableofcontents

\section{Introduction}
The purpose of this paper is to prove the truth in CM fields, of a 
conjecture named after Gross and Kuz'min. We start with some elementary
notation, which allows to expose in Theorem \ref{gc} the precise statement 
of the Conjecture when applied to CM fields.

Let $p$ be an odd prime and $\K \supset \Q[ \zeta ]$ be a galois
extension containing the \nth{p} roots of unity, while $(\K_n)_{n \in
  \N}$ are the intermediate fields of its cyclotomic $\Z_p$-extension
$\K_{\infty}$. Let $A_n = (\id{C}(\K_n))_p$ be the $p$-parts of the
ideal class groups of $\K_n$ and $A = \varprojlim_n A_n$ be their
projective limit. The subgroups $B_n \subset A_n$ are generated by
($p$-free powers of) the classes containing ramified primes above
$p$. If $\wp \subset \K$ is such a prime and $\rg{p} = [ \wp ]$ is its
class, having order $t \cdot p^q$ in the ideal class group, then
$\rg{p}^t \in B(\K)$. We shall assume for simplicity that the orders
of the primes above $p$ in $\K_n$ are all $p$-powers and we let
\begin{eqnarray}
\label{bram}
A'_n  =  A_n/B_n, \quad \quad
B  =  \varprojlim_n B_n, \quad A' = A/B.\nonumber
\end{eqnarray}
The quotients $A'_n$ arise also as ideal class groups of the ring of
$p$-integers $\id{O}(\K)[1/p]$, see \cite{Iw}, \S 4. We let $E_n =
(\id{O}(\K_n))^{\times}$ be the global units of $\K_n$ and $E'_n =
(\id{O}(\K_n)[ 1/p ])^{\times}$ be the $p$-units.

As usual the galois group $\Gamma = \Gal(\K_{\infty}/\K)$
and $\Lambda = \Z_p[ \Gamma ] \cong \Z_p[[ \tau ]] \cong \Z_p[[ T ]]$,
where $\tau \in \Gamma$ is a topological generator and $T =
\tau-1$. With this, the module $A$ is a finitely generated
$\Lambda$-torsion module. We let
\[ \omega_n = (T+1)^{p^{n-1}} - 1 \in \Lambda, \quad \nu_{n+1,n} =
\omega_{n+1}/\omega_n \in \Lambda. \] The groups abelian $A, A', B$ are
endowed with an action of $\Lambda$; we write the action of $\Lambda$ 
multiplicatively, so $a^T = \tau(a)/a$,
etc. It may be useful at places to skip for simplicity of notation to
additive written groups, and this shall be indicated in the text;
moreover, generic, abstract $\Lambda$-modules will be always written
additively, for the same reasons. We use the same notation for the
action of other group rings, such as $\Z_p[ \Delta ], \Z[ \Delta ]$,
etc.

Complex conjugation $\jmath \in \Gal(\K/\Q)$ acts on any module $X$
attached to $\K$ and its extensions, for instance $X = \id{O}(\K_n),
A_n, B_n$ but also the galois groups $\Gal(\KH_n/\K_n)$, etc.,
inducing a decomposition in plus and minus parts: $X^+ = (1+\jmath) X,
X^- = (1-\jmath) X$. Note that $p$ is odd, and since we shall be
concerned only with $p$-groups, division by $2$ is obsolete. If
$\M_n/\K_n$ is a $p$-abelian extension which is galois over $\Q$ and
with $\Gal(\K/\Q)$ acting on $\Gal(\M_n/\K_n)$, then we define
\begin{eqnarray}
\label{hilpm}
 \M_n^+ = \KL_n^{\Gals(\M_n/\K_n)^-}, \quad \M_n^- =
 \M_n^{\Gals(\M_n/\K_n)^+} .
\end{eqnarray}

If $M$ is a Noetherian $\Lambda$-torsion module and $f \in \Z_p[ T ]$
is a distinguished polynomial, we define the $f$-part and the
$f$-torsion of $M$ by
\begin{eqnarray}
\label{parts}
M(f) & = & \{ x \in M \ : \ \exists n > 0 : f^n x = 0 \}, \\ 
M[ f ] & = & \{ x \in M \ : \ f x = 0 \} \subset M(f) .\nonumber
\end{eqnarray}
Since $M$ is finitely generated, there is a minimal $n$ such that $f^n
M(f) = 0$, and we denote this by $\ord_f(M) = n$, the
$f$-order. 


Let $X$ be a finite abelian $p$-group; the constants 
\begin{eqnarray}
\label{sexp}
\exp(X) & = & \max\{ \ord(x) : x \in X\},\\
\sexp(X) & = & \min\{ \ord(x) \ : \ x \in X \setminus p X \} \nonumber
\end{eqnarray}
are the exponent, resp. the \textit{subexponent} of $X$. For
instance, if $X = C_p \times C_{p^3}$, then $\exp(X) = p^3$, but
$\sexp(X) = p$. We have $\exp(X) = \sexp(X)$ iff $X$ is the direct
product of cyclic groups of the same order.

Leopoldt emitted in 1962 the hypothesis that the $p$-adic regulator of
the units $E(\K)$ should be non-vanishing. His initial conjecture
referred to abelian fields $\K$ but it was soon accepted that one
should expect that the same happens in general for arbitrary number
fields $\K$. The statement for abelian fields could be proved in 1967
by Brumer \cite{Br}, using a $p$-adic variant of Baker's fundamental
result on linear forms in logarithms and an earlier argument of Ax
\cite{Ax}. Greenberg showed in 1973 \cite{Gr} how to define the
$p$-adic regulator\footnote{Although Greenberg did not use the term of
  $p$-adic regulator of the $p$-units, his construction was later used
  by Gross for defining this term.} $R(E'(\K))$ of the $p$-units, and
could prove, using a specific $p$-adic argument that allows
application of the theorem of Baker-Brumer, a result on linear forms
in $p$-adic logarithms, to the effect that the regulator $R(E'(\K))$
does not vanish for abelian extension $\K/\Q$. Several years later, in
1981, Federer and Gross \cite{FG} considered the question of the
vanishing of $R(E'(\K))$ for arbitrary CM extensions $\K/\Q$. Unlike
Greenberg, they cannot provide a proof for this assumption; in
exchange, they prove that $R(E'(\K)) \neq 0$ is equivalent to $B^- =
A^-(T)$.  Carroll and Kisilevski gave an example in \cite{CK}, showing
that $(A')^-(T) = \{ 1 \}$ does not hold for $\Z_p$-extensions of
$\K$, different from the cyclotomic.

The description in \cite{FG} yields a useful translation of the
Diophantine statement about the regulator into a class field
theoretical statement about the vanishing of $(A')^-(T)$. Quite at the
same time as Greenberg, and just around the\footnote{iron-} curtain,
L. Kuz'min had formulated in a lengthy paper \cite{Ku} on Iwasawa
theory the {\em Hypothesis H}, which contains the statement $| A'(T) |
< \infty$ for all number fields $\K$. The connection to regulators is
not considered in Kuz'min's paper, but we have here an adequate
generalization of Gross's conjecture to arbitrary number fields
$\K$. In the case of CM fields, the Hypothesis H contains also a
statement $| (A^+)'(T) | < \infty$; this follows also from the
Leopoldt conjecture for $\K^+$.  The conjecture of Leopoldt also has a
simple class field theoretical equivalent, which was proved by Iwasawa
already in 1973, in his seminal paper \cite{Iw}: for CM fields $\K$,
this amounts to the fact that the maximal $p$-abelian $p$-ramified
extension $\Omega(\K^+)$ is a finite extension of $\K^+_{\infty}$.

We stress here the dual Diophantine and class field theoretical
aspects of the conjectures of Gross-Kuz'min and Leopoldt by using the
common term of {\em regulator conjectures of classical Iwasawa
  theory}. In 1986, L.J.Federer undertook the task of generalizing the
classical results of Iwasawa theory \cite{Fe}. These are results on
the asymptotic behavior of $A_n, A'_n$ and Federer considers
generalized class groups of $S$-integers.  She thus considers the
structure of the galois groups of the maximal abelian $p$-extensions
$\M_n/\K_n$ which are ray-class field to some fixed ray, and in
addition split the primes contained in a separate fixed set of places
of $\K$. The paper is algebraic in nature with little reference to the
field theoretic background, but it confirms the general nature of
Iwasawa theory. In this flavor, one may ask in what way the regulator
conjectures of classical Iwasawa theory generalize to Federer's ray
class fields, and whether these generalizations also afford equivalent
formulations, in Diophantine and in class field theoretical forms. It
is likely that one may encounter a proper embedding of Jaulent's
conjecture -- which is a purely Diophantine generalization of the
Leopoldt conjecture, see \cite{Ja} -- in a systematic context of class
field theory.

The purpose of this breve remarks was to situate the questions and
methods that we shall deploy below in their broad context. One can
find in \cite{LMN} or in Seo's recent paper \cite{Seo} a good overview
of further conjectures related to the ones discussed above, as well as
an extensive literature on recent research related to the
Gross-Kuz'min conjecture. Jaulent established connections to
$K$-theory, e.g. in \cite{Ja1}. In this paper we prove the particular
case of the Conjecture, in which $\K$ is a CM extension of $\Q$:

\begin{theorem}
\label{gc}
Let $p$ be an odd prime and $\K$ a CM extension of $\Q$. Then
\begin{eqnarray}
\label{gross} 
 {A'}^-(T) = \{ 1 \}. 
\end{eqnarray}
Here $A'(T) \subset A'$ is the maximal submodule of the Noetherian
$\Lambda$-torsion module $A'$, which is annihilated by some finite
power of $T$.
\end{theorem} 
The fact that $(A')^+(T) \sim B^+ \sim \{ 1 \}$ was established for
arbitrary CM fields in a separate paper concerning the conjecture of
Leopoldt \cite{Mi}.
\subsection{Notations}
\label{nots}
Unless otherwise specified, the fields $\K$ in this paper verify the
following
\begin{assum}
\label{kassum}
The field $\K$ is a galois CM extension of $\Q$ which contains the
\nth{p} roots of unity and such that the primes above $p$ are totally
ramified in the cyclotomic $\Z_p$-extension $\K_{\infty}/\K$ and split
in $\K/\K^+$.  There is an integer $k$ such that $\mu_{p^k} \subset
\K$ but $\mu_{p^{k+1}} \not \subset \K$ and we use the numeration
$\K_1 = \K_2 = \ldots = \K_k \neq \K_{k+1}$; consequently, for $n > k$
we have $[ \K_n : \K_1 ] = p^{n-k}$. The base field will be chosen
such that for all $a = (a_n)_{n \in \N} \in A^-(T) \setminus
(A^-(T))^p$ we have $a_1 \neq 1$. As consequence of the numbering of
fields, we also have $a_1 = a_2 = \ldots = a_k$, etc. There is an
integer $z(a) > -k$ which does not depend on $n$, and such that
$\ord(a_n) = p^{n+z(a)}$ for all $n \geq k$.
\end{assum}
We let $\KH_n \supset \K_n$ be the maximal $p$-abelian unramified
extensions of $\K_n$ -- the $p$-Hilbert class fields of $\K_n$ -- and
$X_n := \Gal(\KH_n/K_n) \cong A_n$, via the Artin Symbol, which we
shall denote by $\varphi$. Let $\KH = \cup_n \KH_n$ be the maximal
unramified $p$-extension of $\K_{\infty}$ and $X =
\Gal(\KH/\K_{\infty})$. The isomorphisms $\varphi: A_n \ra X_n$ are
norm compatible and yield an isomorphism in the projective limit,
which we shall also denote by $\varphi$:
\begin{eqnarray}
\label{plim}
\varphi(A) = \varphi(\varprojlim_n A_n) = \varprojlim_n(\varphi(A_n)) = 
\varprojlim_n(X_n) = X.
\end{eqnarray}

The maximal subextension of $\KH_n$ which splits all the primes above
$p$ is denoted by $\KH'_n \subset \KH_n$ and we have
\[ \Gal(\KH'_n/\K_n) \cong A'_n, \quad \Gal(\KH_n/\KH'_n) = \varphi(
B_n ). \] The inductive limit is $\KH' \subset \KH$, with
$\Gal(\KH/\KH') = \varphi(B)$. (e.g. \cite{Iw}, \S 3. - 4.)
The maximal $p$-abelian $p$-ramified extension of $\K_n$ is denoted by
$\Omega_n = \Omega(\K_n)$ and $\Omega = \cup_n \Omega(\K_n)$. Note
that $\Omega_n \supset \K_{\infty} \cdot \KH_n$, for all $n$. The units and $p$-units of $\K_n$ are
$E(\K_n)$, resp. $E'(\K_n)$ and they generate the following Kummer 
extensions contained in $\Omega_n$:
\begin{eqnarray}
\label{omes}
\Omega_{E}(\K_n) & = & \Omega_n \ \bigcap \ \left( \bigcup_{m \geq n} \K_m[ E(\K_m)^{1/p^m} ] \right), \\
\Omega_{E'}(\K_n) & = & \Omega_n \ \bigcap \ \left( \bigcup_{m \geq n} \K_m[ E'(\K_m)^{1/p^m} ] \right). \nonumber
\end{eqnarray}

We shall make intensive use of the extension 
\begin{eqnarray}
\label{proot}
\KL = \KL(\K) = \bigcup_n \KL_n = \bigcup_n \K_n[ p^{1/p^n} ] = \K_{\infty}[ p\pinf ] \subset \Omega_{E}(\K), 
\end{eqnarray}
the last inclusion being proved in \S 2, \rf{p1m}. We let $\B/\Q_p$ be the $\Z_p$ extension of $\Q$ with intermediate
fields $\B_1 = \Q$ and $\B_n \subset \B, [ \B_n : \Q ] = p^{n-1}$. Over the $p$-adics, we write
$\rg{B}_n = \B_{n,(p)}, \rg{B} = \B_{(p)}$ for the respective completions at the unique prime
above $p$.

The extensions $\KH^-, \Omega^-, (\KH')^-$, etc. are defined according
to \rf{hilpm}. We denote by $N_{m,n}$ the norms $\Norm_{\K_m/\K_n}$,
for $m > n > 0$. In particular, $N_{n,1} = \Norm_{\K_n/\K}$. Note also
that $\K$ being a galois extension, $\Delta = \Gal(\K/\Q)$ acts
transitively upon the primes above $p$. We fix such a prime $\wp
\subset \K$ and a norm coherent sequence of totally ramified primes
$\wp_n \subset \id{O}(\K_n)$ above it, so $N_{m,n}( \wp_m ) = \wp_n$;
the choice is possible since we assumed that the primes above $p$
ramify completely in the $\Z_p$ - cyclotomic extension. Let $b_n = [
\wp_n ] \in B_n$; then the modules $B_n$ are $\Z_p[ \Delta ]$-cyclic,
generated by the classes $b_n$.

\subsection{List of symbols}
We give here a list of the notations introduced below in connection with
Iwasawa theory
\begin{eqnarray*}
\begin{array}{l c l}
p & = & \hbox{An odd rational prime}, \\ 
\zeta_{p^n} & = & \hbox{Primitive
\nth{p^n} roots of unity with $\zeta_{p^n}^p = \zeta_{p^{n-1}}$ for all $n >
0$.},\\ 
\mu_{p^n} & = & \{ \zeta_{p^n}^k, k \in \N \}, \\
\K & = & \hbox{A galois CM extension determined by Assumption \ref{kassum}}, \\
\K_{\infty}, \K_{n}& = & \hbox{The cyclotomic $\Z_p$ - extension of $\K$, and
intermediate fields,}\\ 
\B & = & = \bigcup_n \B_n \hbox{ The $\Z_p$-extension of $\Q = \B_1$},\\
\rg{B} & = & = \bigcup_n \rg{B}_n \hbox{ The $\Z_p$-extension of $\Q_p = \rg{B}_1$},\\
\Delta & = & \Gal(\K/\Q), \quad \hbox{ resp. $\Gal(\rg{K}/\Q_p) \cong D(\wp)$, in the local case}, \\ 
A(\rg{K}) & = &
\hbox{$p$-part of the ideal class group of the field $\rg{K}$}, \\ 
\Gamma & = & \Gal( \K_{\infty}/\K ) = \Z_p \tau, \quad \hbox{$\tau$ a 
topological generator of $\Gamma$},\\
T & = & \tau -1, \\ * & = & \hbox{Iwasawa's
involution on $\Lambda$ induced by $T^* = (p-T)/(T+1)$},\\
\jmath & = & \hbox{The image of complex conjugation
in $\Gal(\K_{\infty}/\Q)$}, \\ 
s & = & \hbox{The
number of primes above $p$ in $\K^+$}, \\ 
C & = & \hbox{Coset representatives for $\Delta/D(\wp)$},\\ 
A'_n  = A'(\K_n) & = & \hbox{The $p$ - part of the ideal
class group of the $p$ - integers of $\K_n$}, \\ 
A' & = & \varprojlim A'_n, \\
B & = & \lan \{ b = (b_n)_{n \in \N} \in A : b_n = [ \wp_n ], \wp_n \supset
(p) \} \ran_{\Z_p}, \\ 
\varphi & = & \hbox{The
Artin symbol, see also \rf{plim} }, \\ 
\KH(\K) & = & \hbox{The maximal $p$ - abelian unramified
extension of $\K_{\infty}$}, \\
\KL(\K) & = & \bigcup_n \KL_n(\K) = \bigcup_n \K_n[ p^{1/p^n} ] \K_{\infty}[ p\pinf ] \subset \Omega_{E}(\K), 
\end{array}
\end{eqnarray*}
\begin{eqnarray*}
\begin{array}{l c l} E_n & = & \id{O}(\K_n)^{\times}, \\
E'_n & = & (\id{O}(\K_n)[ 1/p ])^{\times}, \\
U_n & = & U(\K_n) = \id{O}\left(\K_n \otimes_{\Q} \Q_p\right) = \prod_{\nu \in
  C} U(\K_{n, \nu \wp}), \\
U^{(1)}_n & = & \prod_{\nu \in C} U^{(1)}(\K_{n, \nu \wp}), \\
\Omega(\K) & = & \hbox{The
maximal $p$-abelian $p$-ramified extension of $\K$},\\ 
\Omega_n & = & \Omega(\K_n)  =  \hbox{The
maximal $p$-abelian $p$-ramified extension of $\K_n$},\\ 
\Omega_E(\K_n) & = & \Omega_n \cap \bigcup_{m > n} \K_m[ E(\K_m)^{1/p^m} ], \hbox{ see \rf{omes},} \\
\Omega_{E'}(\K_n) & = & \Omega_n \cap \bigcup_{m > n} \K_m[ E'(\K_m)^{1/p^m} ], \hbox{ see \rf{omes}.}  \\
\end{array}
\end{eqnarray*}

\subsection{Id\`{e}les, completions and their uniformizors}
\label{idels}
In our case $\K$ is galois with group $\Delta = \Gal(\K/\Q)$; with
$\wp \subset \K, \wp_n \subset \K_n$ the fixed primes introduced above,
we denote by $C = \Delta/D(\wp)$ a set of coset representatives of the decomposition
group of $\wp$. If $\wp^+$ is the prime of $\K^+$ above $\wp$, then we
write $C^+ = \Delta^+/D(\wp^+)$. In case that $\wp^+$ splits in
$\K/\K^+$, then $\Delta^+$ acts transitively on the pairs of conjugate
primes above $p$ in $\K$, via the representatives $C^+$. We let $s = |
C^+ |$ be the number of such pairs of complex conjugate primes and $P
= \{ (\nu \wp,\nu \overline{\wp}) \ : \ \nu \in C^+ \}$. Note that if $\wp^+$ is inert in
$\K/\K^+$, then $B^- = \{ 1 \}$ and so is ${A'}^-[ T ]$, as will be
made clear in Lemma \ref{t2} below. Therefore this case is irrelevant
for the Gross conjecture, and we can assume from now on that $\wp^+$
is split in a pair of complex conjugate primes. Let $q = \ord(\wp)
\geq 1$ and $(\rho_0 ) = \wp^q$; we recall that we assumed the order 
of the class $[ \wp ]$ to be a power of $p$: otherwise one can replace
$\wp$ by $\wp^t$, where $t$ is the $p$-free part of $q$.
Then the definition $(\rho) =
(\rho_0/\overline{\rho}_0)$ determines $\rho$ up to roots of unity,
and we choose $\rho = c + O((1-\zeta)^2)$, where $\zeta \in \K$ is a
primitive \nth{p^k} root of unity, and $c$ is
an integer constant coprime to $p$. 
 The classes $b_n := [ \wp_n/\overline{\wp}_n ] \in B_n^-$ form a norm coherent
sequence which we denote by $b = ( b_n )_{n \in \N} \in B^-$ and note
that $b$ generates $B^-$ as a $\Z_p[ \Delta ]$-module.

We let $\eu{K}_n = \K_n \otimes_{\Q} \Q_p = \prod_{\eu{p}}
\K_{n,\eu{p}}$, where $\eu{p}$ runs through all the ramified places
above $p$. If $\eu{p}_n \subset \K_n$ is a sequence of such places
with $\eu{p}_n^{p^{n-k}} = \eu{p} \subset \K$, then $\K_{n,
  \eu{p}_n}/\K_{\wp}$ is an intermediate field of the compositum of
the local cyclotomic $\Z_p$-extension of $\Q_p$ and $\K_{\eu{p}}$. We
denote by $U_n = \id{O}^{\times}(\eu{K}_n)$ the product of the local
units in the various completions.  For $\nu \wp_n \in P_n$ we let
$\iota_{\nu \wp_n} : \eu{K}_n \ra \K_{\nu \wp}[ \mu_{p^n} ]$ be the
natural projection.

Let $\rg{K}$ be an algebra. If $\rg{K} \supset \Q_p$ is a local galois
extension of $\Q_p$ and $U(\rg{K})$ are its units, we denote the
maximal ideal of $\rg{K}$ by $\eu{M}(\rg{K}) \subset \id{O}(\rg{K})$ and
write $\pi = \pi(\rg{K}) \in \eu{M}$ for some uniformizor generating
it. If $\rg{K} = \K \otimes_{\Q} \Q_p = \prod_{\nu \in C} \K_{\nu
  \eu{P}}$ is the $p$-id\`{e}les algebra of some galois number field,
the central ideal is
$ \eu{M}(\eu{K}) = \prod_{\nu \in C} \eu{M}(\K_{\nu \eu{P}})$. If
$\pi_{\nu} \in \eu{M}(\K_{\nu \eu{P}})$ are uniformizors, we say by
extension, that $\pi' := (\pi_{\nu})_{\nu \in C}$ is an (associated)
uniformizor of $\eu{K}$.  In general, an element $(x_{\nu})_{\nu \in
  C} \in \eu{K}_n$ is a uniformizor iff $v_p(x_{\nu}) = v_p(\nu
\wp_{n})$ for all $\nu \in C$.  In particular, we may identify the
uniformizor $\pi_{\nu}$ with $\pi'_{\nu} = (
{\pi'_{\nu}}^{\delta_{\nu, \nu'}} )_{\nu' \in C}$, the vector which is
$1$ at all positions except for $\iota_{\nu}(\pi'_{\nu}) =
\pi_{\nu}$. We assume in \S 3. that a norm coherent set of
uniformizors $\underline{\pi}_n \in \eu{M}(\K_{n, \wp_n}) $ are fixed
and let $\underline{\pi}'_n$ be the associated uniformizor for
$\eu{K}_n$. Let $e(\K)$ be the ramification index of $\wp$ above
$\Q$. Then $\underline{\pi}_n^{e(\K)} \in \eu{M}(\rg{B}_n)$.
 
\subsection{Approach and plan of the proof}
We show at the beginning of Chapter 2. that the assumptions \ref{kassum} are not restrictive
and a proof for fields that verify these assumptions will imply 
the claim of the Theorem \ref{gc} in full generality. Note
that if the primes above $p$ are not split in $\K/\K^+$, then $B^- =
\{ 1 \}$ and the Lemma \ref{t2} below implies that $A^-(T) = \{ 1 \}$,
so this case is trivial.

The proof of Theorem \ref{gc} is based on the following two
facts, which are proved in Chapter 2:
\begin{itemize}
\item[ F1. ] There is an injective map $\hat{T} : (A')^-[ T ] \ra B^-$.
\item[ F2. ] The module $A^+(T^*)$ is finite for every CM extension of $\Q$.
\end{itemize}

Assuming that the Gross-Kuz'min conjecture does not hold in a field $\K$
verifying the assumptions \ref{kassum} and thus $(A')^-[ T ]$ is infinite, 
we raise a contradiction to fact F2. above, as follows. 
\begin{itemize}
\item[ A. ] Following the construction of Greenberg in \cite{Gr} we
  build a family of maps $\psi_{\Pi} : {A'}^-[ T ] \ra U(\K)$ which
  are indexed by sequences of uniformizors, and show eventually that
  $\Pi$ can be chosen, such that $\psi_{\Pi}$ vanishes. This way, the
  map $\wh{T}$ is continued to a map $\wh{\Theta}$ which produces
  radicals of some $p$-ramified extensions.  This defines an
  obstruction-group arising under the premise that the Conjecture is
  false (see also Remark \ref{obstruct}).
\item[ B. ] We consider $\KL = \K_{\infty}[ p\pinf ]$ over which the
  Kummer radicals $\wh{\Theta}({A'}^-[ T ])$ build unramified
  extensions. In Chapter 4. we use this fact for deriving a
  contradiction to F2.
\end{itemize}

In \S 2 we prove the existence of the map $\wh{T}$ and the fact
F2. together with some auxiliary properties of group rings, units and
uniformizors. In \S 3 we generalize the ideas used by Greenberg in his
proof \cite{Gr} of the Conjecture for the case of abelian extensions of
$\Q$, in order to derive some multiplicative maps $\psi_{\Pi}$ as
described above. With the use of these maps we eventually show that
there is a continuation of $\wh{T}$ to maps $\wh{\Theta}$ that produce
radicals of unramified extensions $\F_x = \KL[ \wh{\Theta}(x) ]$ for
$x \in {A'}^-[ T ]$.  In particular, the failure of the Gross - Kuz'min
conjecture is equivalent to the existence of infinite unramified
extensions
\[ \Omega_E(\K) \subset \overline{\F} = \Omega_E(\K)\left[ \wh{\Theta}({A'}^-[ T ])\pinf \right]
 \subset \Omega_{E'}(\K).
\]
In \S 4 we draw a contradiction from the existence of the map
$\wh{\Theta}$. For this, we show that there is a continuation of
$\wh{\Theta}$ to a twisted map of $\Lambda$ -modules $\id{L} : {A'}^-[
T ] \ra A(\KL)$, where $A(\KL)$ is the projective limit of classes in
$A(\KL_n)$, with $\KL_n = \K_n[ p^{1/p^n} ]$. Using galois and class
field theory we prove Proposition \ref{iter}, by means of which it is
possible to lower the classes in $\id{L}_n({A'}^-[ T ]) \subset
A(\KL_n)$ to classes in $A^+(\K_n)$ of unbounded order, thus obtaining
a contradiction to F2., which finally proves Theorem \ref{gc}.

\section{General results}
Let $\K'$ be an arbitrary CM extension of $\Q$. Then there is a finite
algebraic extension $\K/\K'$, which is galois, CM and satisfies all
the conditions for $\K$ which were defined above. For this, take first
$\K''$ as the galois closure of $\K'$ to which we adjoin the \nth{p}
roots of unity. We may choose $\K$ to be some subfield of the
cyclotomic $\Z_p$-extension of $\K''$, in which all the primes above
$p$ are totally ramified. Obviously, $[ \K : \K' ] < \infty$.  Let
thus $\K = \K'[ \theta ]$, as a simple algebraic extension and we let
$\K'_{\infty} = \K' \cdot \B$ be the cyclotomic $\Z_p$-extension of
$\K'$, so $\K_{\infty} = \K'_{\infty}[ \theta ]$. The definition of
$k, \ mu_{p^k} \subset \K$ is obvious.

If the Gross conjecture does not hold for $\K'$, then $(A'(\K'))^-(T)$ is
infinite. The kernel of the ideal lift map $\Ker(\iota: A(\K') \ra A(\K))$ has finite exponent, bounded by $[ \K : \K' ]$,
and lift map commutes with the action of $\Lambda$: therefore $(A'(\K))^-(T)$
must also be infinite. It thus suffices to prove the conjecture for
extensions satisfying the Assumptions \ref{kassum}.

\subsection{Class groups}
In the sequel we shall repeatedly choose representatives $a \in A$ for
the class $a' = a B \in {A'} = A/B$. If $a' \in A(T)$, then 
obviously $a \in A(T)$. The same holds, when restricting to minus parts.

Assuming that the Gross-Kuz'min conjecture is false, we have a natural map
\begin{eqnarray}
 \label{htt}
\wh{T} \ : \ {A'}^-[ T ] \ra B^- \quad \quad \quad a' = a B \mapsto a^T \in B^-.
\end{eqnarray}
The map is well defined, since for $a_i \in A^-, i = 1, 2$ with $a' =
a_1 B = a_2 B$ there is a $b \in B^-$ with $a_2 = b a_1$, so
$\wh{T}(a') = a_1^T = (a_1 b)^T = a_2^T$ does not dependent on the
choice of the representative $a \in a'$.  We consider the interesting
module $B' = \wh{T}\left((A')^-[ T ]\right)$.  The next lemma
generalizes a fact noted in the abelian case by Greenberg, and which
amounts to the fact that $B' \cong {A'}^-[ T ]$:
\begin{lemma}
\label{t2}
Let $\K$ be a CM extension of $\Q$ and suppose that 
${A'}^-(T) \neq \{ 1 \}$. Then the map $\wh{T}$ is injective and
the module 
\[ B' = \wh{T}\left((A')^-[ T ]\right) \cong {A'}^-[ T ]. \]
\end{lemma}
\begin{proof}
  Let $a' = a B, a = (a_n)_{n \in \N} \in A^-$ and let $\eu{Q} \in
  a_n$ be a prime ideal, let $n$ be sufficiently large and $\ord(a_n)
  = p^{n+z}$, for some $z \in \Z$.  Consider the principal ideal
  $(\alpha_0) = \eu{Q}^{p^{n+z}}$ and $\alpha =
  \alpha_0/\overline{\alpha_0}$. Since $a' \in {A'}^-[ T ]$, it also
  follows that $a_n^T \in B_n^-$ and thus $\eu{Q}^T = \eu{R}_n$ with
  $c_n := [ \eu{R}_n ] \in B_n^-$. If $c_n \neq 1$, then we are done.

  We thus assume that $c_n = 1$ and derive a contradiction. In this
  case $\eu{R}_n^{1-\jmath} = (\rho_n)$ is a $p$-unit and
  $(\alpha_0^T) = (\rho_n^{p^{n+z}})$, so
  \[ \alpha^T = (\delta/\overline{\delta}) \cdot \rho_n^{(1-\jmath) p^{n+z}},
  \quad \delta \in E(\K_n). \] It follows from Kronecker's unit
  theorem that $\varepsilon = \delta/\overline{\delta} \in \mu_{2
    p^n}$ and by taking the norm $N = N_{\K_n/\K}$ we obtain 
$N(\varepsilon) N(\rho_n)^{(1-\jmath) p^{n+z}} \in \mu(\K)$. It follows that
  $\rho_1 := N(\rho_n^{1-\jmath})$ verifies $\rho_1^{p^{n+z}} =
  \varepsilon_1 \in \mu_{2 p^k}$. Thus $\rho_1^{p^k} = \pm 1$. By
  Hilbert 90 we deduce that $\rho_n^{(1-\jmath) p^k} = \pm x^T, x \in
  \K_n^{\times}$. In terms of ideals, we have then
\begin{eqnarray*}
\eu{Q}^{(1-\jmath) T p^{n+z}} & = & (\alpha^T) = (x^{T p^{n+z-k}}), \quad
\hbox{hence} \\
\left(\eu{Q}^{(1-\jmath) p^k}/(x)\right)^{T p^{n+z-k}} & = & (1) \quad
\Rightarrow \quad (\eu{Q}^{(1-\jmath) p^k}/(x))^T = (1).
\end{eqnarray*}   
The ideal $\eu{B} = \eu{Q}^{(1-\jmath) p^k}/(x) \in a_n^{2 p^k}$
verifies herewith $\eu{B}^T = (1)$. Since $a_n \in A'_n$, it follows that
$\eu{B} = \id{O}(\K_n) \eu{B}_1$ is the lift of an ideal $\id{B}_1
\subset \K$. But then $\ord(a'_n) \leq p^k \exp(A_1^-)$, so the orders
of $a'_n$ are uniformly bounded, which is impossible, since $a'_n \in
A_n^-$. Consequently, for $a' \neq 1$ we have $\wh{T}(a') \neq 1$ and
the map is injective. The claim $\wh{T}({A'}^-[ T ]) \cong {A'}^-[ T ]$
follows, thus completing the proof.
\end{proof}

The next result is a common fact in class field theory which will be
used for deriving the final contradiction for the proof of the Gross
conjecture. We provide its proof, for the sake of completeness.
\begin{lemma}
\label{noT*}
Let $\K$ be a CM extension of $\Q$. Then $A^+(T^*)$ is a finite
module.
\end{lemma}
\begin{proof}
  Suppose that $A^+(T^*)$ is infinite. Then
  $\Gal(\Omega^+/\K_{\infty})(T^*)$ is a fortiori infinite and since
  $E^-(\K_n) = \mu(\K_n)$, we have $\Omega^+ = \K_{\infty}[ (A^-)\pinf
  ]$. Let $\id{B} \subset A^-( T )$ be the submodule of sequences of
  classes with $\K_{\infty}[ b\pinf ] \subset \KH^+$; it follows that
  $\id{B}(T)$ is infinite too, being -- up to possible finite cokernel
  -- the radical of the maximal subextension $\KL \subset \KH^+$ with
  $\Gal(\KL/\K_{\infty})^{{T^*}^m} = \{ 1 \}$ for some $m > 0$.

  If $B^- \cap \id{B} = \{ 1 \}$, then the previous lemma implies that
  $\id{B} \cap {A'}^-(T) = \id{B} = \{ 1 \}$. Therefore, it suffices
  to prove that $B^- \cap \id{B} = \{ 1 \}$.  We assume thus that $B^-
  \cap \id{B} \neq \{ 1 \}$ and raise a contradiction. Let $\wh{b} =
  (\wh{b}_n)_{n \in \N} \in B^- \cap \id{B}$ and $\wp_n \subset \K_n$
  be a ramified prime above $\wp \subset \K$ and let like in \S
  \ref{idels} $(\rho_0) = \wp^q$. There is a constant $z \in \Z$ such
  that $\ord(\wp_n) = p^{n+z}$ for all sufficiently large $n$; we let
  $m = \max( 0, -z )$, and $n$ be large, so $\wp_n^{p^{n+m}} :=
  (\varpi)$ is a principal ideal. We have $\wp_n^{(1-\jmath) p^{n+m}}
  = (\varpi/\overline{\varpi})$, with $(\varpi) = (\wp^{p^{m+k}})$,
  which is by construction a principal ideal: it follows thus for the
  constant $m$ that $p^{m+k} \geq q$.  Consider now the cyclotomic
  $\Z_p$-extension $\rg{K}_{\infty}/\rg{K}$ of the completion $\rg{K}
  = \K_{\wp}$ and let $t_n \in \Z[ C^+ ]$ be such that $\wp_n^{
    (1-\jmath) t_n} \in \wh{b}_n$.  By choice of $\wh{b}_n$, we must
  have $U(\K_n)^{\ord(\wh{b}_n)} \cap (\varpi^{(1-\jmath) t_n}) \neq
  \emptyset$. Since the primes above $p$ are relatively coprime, there
  is a constant $c$ and a unit $u \in U(\rg{K})$ such that
  $\iota_{\wp}(\varpi^{(1-\jmath) t_n}) = u \pi^{q c}$, with $\pi \in
  \id{O}(\rg{K})$ being a uniformizor. We may assume that $u$ is
  chosen such that $u \pi^{q c} \in \rg{K}_n = \rg{K}[ \zeta_{p^n} ]$
  and $(c,p) = 1$. We may even assume that $u = 1$: if not, we adjoin
  $\rg{K}_n' = \rg{K}_n[ u^{1/q c} ]$ and if $u \pi^{q c}$ is a $q
  p^{n-k}$-th power in $\rg{K}_n$, then it also is one in
  $\rg{K}'_n$. But then there is a uniformizor $\pi' = u^{1/qc} \pi
  \in \rg{K}'$ such that ${\pi'}^{q c} = x^{q p^{n-k} }$ for some $x
  \in \rg{K}'_n$. However, $\rg{K}'_n = \rg{K}'[ \zeta^{1/p^{n-k}} ]$
  and since $(c, p) = 1$, by letting $y = x^{q}$, it follows from
  Kummer theory that $y = {\pi'}^{1/p^{n-k}}$. Consequently
\[ \rg{K}'[ \zeta^{1/p^{n-k}} ] = \rg{K}'[ {\pi'}^{1/p^{n-k}} ] , \]
and $\pi' \zeta^a \in ({\rg{K}'}^{\times})^{p^{n-k}}$, so $v_p(\pi') \leq \frac{1}{p^{n-k}}$. 
For $n \ra \infty$ we obtain $v_p(\pi') = 0$, which is absurd. 
Therefore, the assumption must be
false and $B^- \cap \id{B} = \{ 1 \}$, which completes the proof.
\end{proof}
\begin{remark}
\label{Jau}
F. Jaulent asked the following question: is $B^- \cap A^p \subset
(B^-)^p$?  Assume that $x \in A^-$ with $x^p \in A^p \cap B^-$; we
have $x^{T p} = 1$, so either $x^T = 1$ or $\ord(x^T) = p$ and
$\Lambda x^T \cong \Lambda/p \Lambda$, since $\Lambda x^T \subset A^-$
is infinite. In particular, if $\mu^-(\K) = 0$ then $B^- \cap A^p =
(B^-)^p$.  We prove that $\mu(\K) = 0$ for CM fields in a submitted
paper, so the question has an affirmative answer.

{\tiny We can show that $B' \cap A^p \subset (B')^p$ without using
  $\mu = 0$. For this let $A^-_{\mu}$ be the $\Z_p$-torsion of $A^-$
  and $A^-_{\lambda} = \{ x \in A^- \ : \ \prk(\Lambda x) < \infty \}$
  and $D = A^-_{\lambda} \oplus A^-_{\mu}$. The pseudoisomorphism of
  $A^-$ to an elementary $\Lambda$-module implies that $A^-/D$ is
  finite, so there is some $n > 0$ such that $\omega_n A^- \subset
  D$. By eventually redefining the base field $\K$, we may thus assume
  that $T A^- \subset D$. Consider now $x \in {A'}^-[ T ]$ with
  $\wh{T}(x)^p \in B^-$. Then $x^T = y_l + w_m$ with $y_l \in
  A^-_{\lambda} \cap B^-$ and $w_m \in A^-_{\mu}$.  But since $x^T \in
  B' \subset B^-$, it follows that $w_m = 1$, so we must actually have
  $x^T \in B^-$ as claimed. However, for proving that $B^- \cap A^p
  \subset (B^-)^p$ we need $\mu = 0$.  }

Finally, we show that for $a \in A[ T^{p-2} ]$ we have $\ord(a_n) = p
\ord(a_{n-1})$ for all $n > 1$. Indeed, since the ideal lift map is
injective on the minus part, we have
\[ \ord(a_{n-1}) = \ord(N_{n,n-1}(a_n)) = \ord\left(\left(p
    f_n(\omega_{n-1}) + \omega_{n-1}^{p-1})\right) a_n\right) = \ord(p
a_n), \] as claimed.
\end{remark}
Next, we investigate the group structure of $B^-$ and $B'$ as $\Z_p[
\Delta^+ ]$-modules; for this we refer to the definitions of $\wp_n,
b_n$ and $C^+$, etc. given in \S \ref{nots}.
\begin{lemma}
\label{bstruct}
The module $B^-$ is spanned by the classes $\ \{\nu b : \nu \in C^+\}$
as a free $\Z_p$-module and $B^- = \Z_p[ \Delta ] b$ is a cyclic
$\Z_p[ \Delta ]$-module of $\Z_p$-rank $s = | C^+ |$. Moreover, there
is a $t_0 \in \Z_p[ C^+ ]$ such that $b^{t_0}$ generates $B'$ as a
$\Z_p[ C^+ ]$-module.
\end{lemma}
\begin{proof}
  Consider the ideals $\eu{R}_n = \wp_n^{1-\jmath}$ and their classes
  $b^2_n = [ \eu{R}_n ] \in B_n^-$; then it follows from the
  definition of $B^-$ that $C^+ b = \{ \nu b \ : \ \nu \in C^+\}$
  generate $B^-$ as a $\Z_p$-module. Since $b^T = 1$, the structure of
  $B^-$ as $\Z_p$- and as $\Lambda$-modules coincide. It remains to
  show that the $C^+ b$ are linearly independent over $\Z_p$. If this
  were not the case, then there is a $\Z_p$-linear dependence in $B^-$
  and we can assume without loss of generality that
  \[ b^{n_1} = \prod_{\nu \in C^+; \nu \neq 1} \nu (b^{n_{\nu}}) =
  b^{\theta}, \quad n_{\nu} \in \Z_p . \] Let $n > 0$ be fixed and
  $z_{\nu} \in \Z$ approximate $n_{\nu}$ to the power $q
  p^{n-k}$. Then
\[ \wp_n = (x) \cdot \prod_{\nu \in C^+; \nu \neq 1} \nu
(\wp_n^{z_{\nu}})  \cdot \eu{Q}, \quad x \in \K_n,  \]
for some ideal $\eu{Q} \subset \id{O}(\K)^{1+\jmath}$. 
By applying $(1-\jmath)$ to the above identity, we obtain
\[ \eu{R}_n = (x/\overline{x}) \cdot \prod_{\nu \in C^+; \nu \neq 1}
\nu (\eu{R}_n^{z_{\nu}}), \] and by acting with $T$ in the above, we
find $(x^T/\overline{x})^T =: \delta \in \mu_{p^n}$.  Since
$N_{n,1}(\delta) = 1$, and $\mu_p \subset \K$, we have a fortiori
$\delta \in \mu_{p^n}^T$. There is thus $\xi \in \mu_{p^n}$ such that
$\delta = \xi^{(1-\jmath) T}$, and consequently $ (x /\xi)^{(1-\jmath)
  T} = 1$; letting $x_0 = \overline{\xi} x$ and $y_0 =
x_0/\overline{x}_0 \in \K$, we have
\[ \eu{R}_n = (y_0) \cdot \prod_{\nu \in C^+; \nu \neq 1} \nu
(\eu{R}_n^{z_{\nu}}). \] The primes above $p$ are pairwise coprime, so
taking the valuation at $\wp_n$ in the previous identity leads to:
$v_{\wp_n}(y_0) = v_{\wp_n}(x_0) = 1$.  Since $y_0 \in \K$, it follows
that $x_0 \in \wp$, and thus $v_{\wp}(x_0) \geq 1$. But the primes
above $\wp$ are totally ramified, so we obtain the following
contradiction
\[ 1 = v_{\wp_n}(x_0) = p^{n-k} v_{\wp}(x_0) , \quad \hbox{hence $ 1
  \leq v_{\wp}(x_0) = 1/p^{n-k} $ }. \] It follows that $\nu b_n$ are
indeed linearly independent, which completes the proof of the first
claim. Since $B' \subset B^-$ and the latter module is $C^+$-cyclic by
$b$, it follows that so must be $B'$ and it has a generator which can
be expressed as $b' = b^{t_0}$ for some $t_0 \in \Z_p[ C^+ ]$.
\end{proof}

\subsection{Norm residues and local uniformizors}
We consider the set of formal $\Z_p$-linear combinations
\begin{eqnarray}
\label{zpc}
\eu{ C } = [ C^+ ]_{\Z_p} = \left\{ t = \sum_{\nu \in C^+ } c_{\nu} \nu \ : \ c_{\nu} \in \Z_p \right\} \subset 
\Z_p[ \Delta ].
 \end{eqnarray}
 This is a $\Z_p$-submodule of $\Z_p[ \Delta ]$, but it is in general
 not a ring. If $M$ is a $\Z_p[ \Delta ]$-module and $x \in M$, we
 shall denote the action of $\eu{C}$ by
 \[ \eu{C} x = \left\{ t x = \prod_{\nu \in C^+} \nu(x)^{c_{\nu}} \ :
   \ t = \sum_{\nu \in C^+} c_{\nu} \nu \in \eu{C} \right\} \subset M.
\]
If $x$ is in addition fixed by $D(\wp)$, then $\eu{C} x \subset M$ is
a canonical module in the sense that it does not depend on the choice
of $C^+$. Throughout the rest of this paper, $\eu{C}$ will be applied to
$D(\wp)$-invariant elements.  As an important instance, $B^- = \eu{C}
b$ in a canonical way, since $b$ is fixed by $D(\wp)$; the Lemma
\ref{bstruct} implies that in fact $C^+ b$ build a base for the
$\Z_p$-module $B^- $.

We let $\rho \in \K^{1-\jmath}$ be fixed, such that
$\wp^{(1-\jmath) q} = \eu{R}^q = (\rho)$: for $e \in \Z$ divisible by
the order $w = | W |$ of the group of roots of unity $W \subset E(\K)$, 
$\rho^e$ is uniquely defined by $\wp$. We also have:
\begin{eqnarray}
\label{rho}
\left(\eu{R}_n^{p^{n-k}}\right)^q = \eu{R}^q = (\rho), \quad \forall n > 0. 
\end{eqnarray}
If $c_n = b_n^{\theta} \in B_n^-$ is an arbitrary class of $B_n^-$, with
$\theta = \sum_{\nu \in C^+} c_{\nu} \nu \in \eu{C}$, and for $\eu{B}
\in c_n$ an arbitrary ideal, we have the following useful
relation:
\begin{eqnarray}
\label{bn}
\eu{B}^{q p^{n-k}} = (\gamma^{q p^{n-k}} \rho^{\theta} ), \quad 
\gamma \in \K_n^{1-\jmath} \quad
\hbox{ with } \quad \eu{B} = (\gamma) \eu{R}_n^{\theta}.
\end{eqnarray}

We next investigate some useful norm coherent sequences of
uniformizors in local fields. Let $\rg{K}_m = \Q_p[ \mu_{p^m} ]$ be
the \nth{p^m} cyclotomic extension of $\Q_p$ and $\rg{B}_m \subset
\rg{K}_m$ be the subfield fixed by the unique cyclic subgroup of order
$p-1$, as defined in the introduction. The numbers $1-\zeta_{p^m} \in
\rg{K}_m$ are uniformizors with the property of being norm coherent
with respect to the norms $N_{m',m} : \rg{K}_{m'} \ra \rg{K}_m$. Then
$\varsigma_m := \Norm_{\Q_p[ \zeta ]/\Q_p}(1-\zeta_{p^m})$ form a
fortiori a norm coherent sequence of uniformizors for $\rg{B}_m$, with
$\varsigma_1 = p$.  Let $\delta_m = \varsigma_p^T \in
\id{O}^{\times}(\B_m)$; then $\delta_m$ form a norm coherent sequence
of cyclotomic units with $N_{m,1} (\delta_m) = 1$. Moreover, $\delta_m
\not \in \rg{B}_m^p$; in order to see this, we consider the (global)
cyclotomic $\Z_p$-extension $\B/\Q$ with $\B_m = \B \cap
\rg{B}_m$. Since $\delta_m \in E(\B_m)$, the assumption $\delta_m \in
\rg{B}_m^p$, would imply that $\B_m[ \zeta, \delta_m^{1/p} ]$ is an
unramified extension of $\Q[ \zeta_{p^m} ]$. Kummer duality leads to a
contradiction to Herbrand's Theorem (\cite{Wa}, p. 100). Let now
$\rg{L} \supset \Q_p$ be any finite extension containing the \nth{p}
roots of unity, let $\rg{L}_m = \rg{L} \cdot \rg{K}_m$ and let $e$ be
the ramification index in $\rg{L}/(\rg{L} \cap \rg{B}_{\infty})$; this
ramification index is constant for all extensions $\rg{L}_m/(\rg{L}_m
\cap \rg{B}_{\infty})$. If $\eu{M}_m \subset \rg{L}_m$ is the maximal
ideal, then $\varsigma_m \in \eu{M}_m^e$ for all sufficiently large
$m$. Note the identity $N_{\B_m/\Q} = p^{m-1} + T f(T) \in \Z[ T ]$,
which holds for some (distinguished) polynomial $f \in \Z[ T ]$. It
implies that
\[ p = N_{\B_m/\Q}(\varsigma_m) = \varsigma_{m}^{p^{m-1}} \cdot
\delta_m^{f(T)} ; \] In particular, $p^{1/p^m} \in \B_{m+1} \cdot
\delta_{m+1}^{1/p^m} \subset \Omega_E(\B_{m+1})$.

We note for future reference:
\begin{lemma}
\label{locuni}
If $\rg{B} = \cup_m \rg{B}_m$ is the $\Z_p$-cyclotomic extension of
$\Q_p$, then
\[ \varsigma_m = N_{\rg{B}_m[ \zeta ]/ \rg{B}_m} (1-\zeta_{p^m}),
\quad \delta_m = \varsigma_m^T , \] are norm coherent sequences of
(global) uniformizors, resp. units. Moreover $\delta_m \not \in
\rg{B}_m^p$ for all $m$ and
\begin{eqnarray}
\label{p1m}
p^{1/p^m} \in \B_{m+1}[ \delta_{m+1}^{1/p^m} ] \subset \Omega_E(\B_{m+1}) .
\end{eqnarray}
If $\K$ is a global field containing the \nth{p} roots of unity and
$\K_{\infty} = \cup \K_n$ is its cyclotomic $\Z_p$-extension, then the
fields $\KL_n := \K_n[ p^{1/n} ] \subset \Omega_E$ and in particular
\begin{eqnarray}
\label{rootp}
\KL := \cup_n \KL_n = \K_{\infty}[ p\pinf ] \subset \Omega_E. 
\end{eqnarray}
\end{lemma}
Let $\rg{K}$ be a local galois extension of $\Q_p$ containing the
\nth{p} roots of unity and $\rg{K}_n = \rg{K}[ \mu_{p^n} ]$ while
$\rg{K}_{\infty} = \cup_n \rg{K}_n$; we denote by $\Delta$ the galois
group $\Gal(\rg{K}/\Q_p)$.  Let $\rg{N} := N_{\infty} = \cap_n N_{n,1}
(\rg{K}^{\times}_n)$.  We let $\tilde{\Delta} \subset
\Gal(\rg{K}_{\infty}/\Q_p)$ be a set of lifts of $\Delta$ and first
give some simple properties of the norm defect:
\begin{lemma}
\label{ndef}
Notations being like above
\begin{itemize}
\item[ 1. ] If $x \in \rg{N}$ and $\sigma \in \Delta$ then
  $(\sigma(x))^{\Z_p} \subset \rg{N}$.
\item[ 2. ] If $x \not \in \rg{N}$, then $(\sigma(x))^{\Z_p} \cap
  \rg{N} = \{ 1 \}$ for all $\sigma \in \Delta$.
\item[ 3. ] If $\rg{K} = \rg{B}[ \zeta_p ]$ is the $\Z_p$-extension of
  $\Q_p$ then $\rg{N} \cap \Z^{\times}_p = \{ 1 \}$.
\end{itemize}
\end{lemma}
\begin{proof}
  Let $\sigma \in \Delta$ and $\tilde{\sigma} \in \tilde{\Delta}$ be a
  lift.  Then $x \in \rg{N}$ iff $\sigma(x) \in \rg{N}$. Indeed,
  suppose that $x \in \rg{N}$ and let $(x_n)_{n \in \N}$ be a sequence
  with $x_n \in \rg{K}_n$ and $N_{n,1}(x_n) = 1$.  Then
  $\tilde{\sigma}(x_n)$ is a sequence with
  $N_{n,1}(\tilde{\sigma}(x_n)) = \sigma(x)$ and conversely, which
  confirms the claim.

  Next we claim that $x \in \rg{N}$ iff $x^c \in \rg{N}$ for all $c
  \in \Z_p$; consider first $x \in \rg{N}$ and a sequence $(x_n)_{n
    \in \N}$ like before. Then $N_{n,1}(x_n^c) = x^c$ and thus $x^c
  \in \rg{N}$, for all $c \in \Z_p$. Conversely, let $c = p^m c_0 \in
  \Z_p, m \geq 0$ and suppose that $x^c \in \rg{N}$.  If $m = 0$, then
  $c^{-1} \in \Z_p$ and by applying the previous fact, we deduce that
  $x \in \rg{N}$. We can thus assume that $c = p^m$ and $x^{p^m} \in
  \rg{N}$. Let $\overline{x} := x \rg{N}$ be the class of $x$ in
  $\rg{K}^{\times}/\rg{N}$. By local class field theory, the quotient
  is $\Z_p$-cyclic, being isomorphic to $\Gamma
  =\Gal(\rg{K}_{\infty}/\rg{K})$. But if $x^{p^m} \in \rg{N}$, we
  would have $\overline{x}^{p^m} = 1$ in contradiction to the fact
  that $\Z_p$ has no elements of finite order. We have shown that $x
  \in \rg{N}$ iff $x^{\sigma \Z_p} \subset \rg{N}$ for all $\sigma \in
  \Delta$.

  It follows from the above proof that $x \not \in \rg{N}$ iff
  $\sigma(x) \not \in \rg{N}$ for all $\sigma \in \Delta$. We have
  also shown that $x^c \in \rg{N}$ iff $x \in \rg{N}$; consequently,
  if $x \not \in \rg{N}$, then $x^{\Z_p} \cap \rg{N} = \{ 1 \}$. This
  completes the proof of 1. and 2.

  For point 3. we note that $\rg{B}/\Q_p$ being a totally ramified
  extension, we have $\Gamma \cong U^{(1)}(\Z_p)/\left(\rg{N} \cap
    U^{(1)}(\Z_p)\right)$. Since $\Gamma \cong U^{(1)}(\Z_p) \cong
  \Z_p$, the previously proved facts imply that $\rg{N} \cap
  U^{(1)}(\Z_p) = \{ 1 \}$. By adjoining the \nth{p} roots of unity,
  if follows that the \nth{p-1} roots of unity in $\Z_p^{\times}$ are
  not norms, so $\Z^{\times}_p \cap \rg{N} = \{ 1 \}$, as claimed.
\end{proof}

The following description of the norm defect will guide our 
investigation of the maps to be introduced in the next Chapter:
\begin{lemma}
\label{deco}
Notations being like above, let $\rg{K}/\Q_p$ be an arbitrary finite
galois extension which contains the \nth{p} roots of unity, and
$\rg{N} = N_{\infty}$. We denote by $\varphi : \rg{K}^{\times} \ra
\Gamma$ the local Artin symbol. Then
\begin{itemize}
\item[ A. ] We have $\Z_p^{\times} \cap \rg{N} = \{ 1 \}$.
\item[ B. ] The exponent of the quotient $U(rg{K})/(\Z_p^{\times}
  \cdot (\rg{N} \cap U(\rg{K})))$ is finite. In particular, there is
  a smallest $p$-power $v = p^c$, such that $V := U(\rg{K})^v$ allows the direct
  product decomposition $V = (\Z_p^{\times} \cap V) \cdot (\rg{N} \cap
  V)$.
\item[ C. ] If $v = 1$, then we have the following direct sum in
  additive notation: $U(\rg{K}) = \Z_p^{\times} \oplus (\rg{N} \cap
  U(\rg{K}))$. Otherwise, there is a set of coset representatives $Z
  \subset U^{(1)}(\rg{K})$ of $\varphi(\Gamma)$, with $U^{(1)}(\Z_p)
  \subset Z$ and $[ Z : U^{(1)}(\Z_p) ] = v$, such that $Z \oplus
  (\rg{N} \cap U^{(1)}(\rg{K})) = U^{(1)}(\rg{K})$.
\end{itemize}
\end{lemma}
\begin{proof}
  Since $\rg{B}[ \zeta_p ] \subset \rg{K}_{\infty}$ by definition, the
  point A. follows from point 3. in the previous lemma. If $\varphi$
  is the local Artin symbol, then
  \[ \Z_p \hookrightarrow \varphi^{-1}(\Gamma) \cong U(\rg{K})/(\rg{N}
  \cap U(\rg{K})) , \] and we have an exact sequence $1 \ra \Z_p \ra
  U(\rg{K})/(\rg{N} \cap U(\rg{K})) \ra K \ra 1$ with finite cokernel
  $K$. If $v$ annihilates $K_2$ and $V := U(\eu{K})^v$, then switching
  to additive notation, we obtain the isomorphism
  \[ V = (\Z_p^{\times} \cap V) \ \oplus \ (\rg{N} \cap V), \] which
  completes the proof of point B, of which C. is a reformulation.  We
  note that it is indeed possible to have $v > 1$
\end{proof}

Finally, we consider the case of a global extension $\K$ verifying the
Assumptions \ref{kassum} and let $\wp \subset \K$ be a prime above
$p$, with $\rg{K} \cong \K_{\wp}$ the local field obtained by
completion of $\K$ at the prime $\wp$ -- thus maintaining the notation
introduced in \S \ref{nots}.  Let $s$ denote the number of pairs of
conjugate primes above $p$ and $C \subset \Gal(\K/\Q)$ act
transitively on these primes. We let
\[ \eu{Z}_0 = \prod_{\nu \in C} U(\Q_p), \quad \eu{Z} = \prod_{\nu \in
  C} Z(\K_{\nu \wp}), \] where the various copies of the units
$U(\Q_p)$ are identified with submodules of the embeddings $\K
\hookrightarrow \rg{K}$ induced by completion at the prime $\nu \wp$;
the modules $Z(\K_{\nu \wp})$ are the coset representatives defined in
Lemma \ref{deco}. By definition, $\eu{Z}_0$ and $\eu{Z}$ are $C^+$
invariant.

Let $N_{\infty} \subset \rg{K}$ be defined like in the above lemma and
$\id{N} = \prod_{\nu \in C} (N_{\infty, \nu} \cap U^{(1)}(\K_{\nu
  \wp}))$, where $N_{\infty, \nu}$ are identified with submodules of
embeddings of $\K$, as in the case of $\eu{Z}$. Let $V_{\nu} =
U^{(1)}(\K_{\nu \wp}))^v$ with $v$ determined in point C of Lemma
\ref{deco} and $V = \prod_{\nu \in C} V_{\nu}$. The same statement in
Lemma \ref{deco} yields the global identities:
\begin{eqnarray}
\label{gldeco}
U^{(1)}(\eu{K}) = \id{N} \oplus \eu{Z}, \quad \quad \quad V = (\id{N} \cap V) \oplus (\eu{Z}_0 \cap V).
\end{eqnarray}
Note that this decomposition can be obtained for all $\K_n$ as base
fields. Moreover, if $N_{\infty,n}$ is defined in the natural way, we
have $N_{n,1}(N_{\infty,n}) = N_{\infty}$.

\section{The maps $\psi$}
Let $\K$ verify Assumption \ref{kassum}, fix a prime $\wp \subset \K$
above $p$ and let $C, C^+$ be defined in \S \ref{idels} -- we refer
the reader to th various additional notations introduced there, which
shall be applied in the sequel. We assumed that the Gross-Kuz'min
Conjecture fails for $\K$ and noted in \S \ref{idels} that complex
conjugation $\jmath \in C$, since otherwise the conjecture is
trivially true.

We have shown in Lemma \ref{t2} that
 $\wh{T} : {A'}^-[ T ] \hookrightarrow B'$ is an isomorphism of cyclic 
$\Lambda$-modules. It
will be convenient to fix a generator $a' \in {A'}^-[ T ]$ and $\theta
\in \Z_p[ C ]$ such that $\wh{T}(a') = b^t$. For $b \in B^-$ 
fixed in the introduction, we have maps
\begin{eqnarray}
\label{thetas}
\begin{array}{l c l l l}
\Theta & : & {A'}^-[ T ] \ra \eu{C} & a' \mapsto x & \hbox{ with $\wh{T}(a') = b^x$} \\ 
{ } & & { } & & \\
\Theta_n &  : & ({A'}^-[ T ])_n \ra \Z[ C ] & a' \mapsto t & \hbox{ with $\Theta(a') -t \in ( \ord(b_n) )$}. 
\end{array}
\end{eqnarray}
Note that $\Theta_n$ is defined up to multiples of $\ord(b_n)$.

After eventually raising to a constant power, say $r(\K)$, we can
continue the map $\wh{T}$ to a map that produces radicals of Kummer
extensions,
\begin{eqnarray}
\label{eval}
\wh{\Theta} & \ : \ & {A'}^-[ T ] \ra \prod_{\nu \in C} (p^{\Z_p} , 
\quad a' \mapsto \rho^{r(\K) \Theta(a')} \in \wh{\eu{K}}, \\
\wh{\Theta}_n & \ : \ & ({A'}^-[ T ] )_n \ra \prod_{\nu \in C} (p^{\Z} \cdot (\eu{\K}^{\times})^{p^n}, \
\quad a' \mapsto \rho^{r(\K) \Theta_n(a'_n)} \nonumber
 \end{eqnarray}
 which is an homomorphism of $\Lambda$-modules; the completion
 $\eu{K}$ will be formally defined below and it allows building
 projective limits $p^x, x \in \Z_p$. In the rest of this chapter, we
 investigate the action of $\Theta$ on $\rho$, defined by \rf{rho} and
 aim to prove the existence of the map $\wh{\Theta}$ and in particular
 the fact that $\iota_{\nu}(\rho^{r(\K) \Theta_n(a'_n)}) \in p^{\Z}
 \cdot (\K^{\times})^{\ord(b_n)}$.

\subsection{Pseudouniformizors}
The projections $\iota_{\nu} : \eu{K}_n \ra \K_{n, \nu \wp}$ yield the
\textit{components} of elements $x \in \eu{K}_n$ in the various
completions and we shall use no index $n$ for the projection, the
index being determined by the context. Note the following
\begin{fact}
\label{exp}
There is some constant exponent $r(\K)$ with the property that, for
all $n \in \N$ we have
\begin{itemize}
\item $\mu(\K)^{r(\K)} = \{ 1 \}$, where $\mu(\K)$ are the $p$-power
  roots of unity in $\K$. Moreover $r(\K) \geq q$ and $\wp^{(1-\jmath)
    r(\K)} = (\rho^{r'(\K)})$ with $r'(\K) = r(\K)/q$.
\item We have $p \in \eu{M}^{r(\K)}(\K_{\wp})$ and for every $n \in
  \N$, any $\nu \in C$ and any uniformizor $\pi \in \K{n, \nu \wp}$ we
  have $\varsigma_n \in \eu{M}^{r(\K)} \K{n, \nu \wp}$.
\item The constant $v$ introduced in Lemma \ref{deco} divides $r(\K)$.
\end{itemize}
\end{fact}
\begin{proof}
  The fact follows by choosing $r(\K)$ as the least common multiple of
  several constants. Some of them are $q, v$ and the ramification
  index $e(\K_{\wp})$. Moreover, $r(\K)$ must be divisible by $w =
  p^k$, the order of the $p$-roots of unity in $\K$. One verifies that
  the least common multiple of these constants fulfills the claims.
\end{proof}

Let now $\rg{K}_n = \K_{n, \wp_n}$ be a finite extension of $\Q_p$ and
$\pi_n \in \rg{K}_n$ be a uniformizor for $\eu{M}(\rg{K}_n)$. Then for
all $x \in \rg{K}_n$ there are unique $a \in \Z$ and $u \in
\id{O}^{\times}(\rg{K}_n)$ such that $x = \pi_n^a \cdot u$. We extend
this decomposition to $\eu{K}_n$ as follows: let $\pi'_n = (\pi_n, 1,
\ldots, 1)$ in the usual Chinese Remainder decomposition of
$\eu{K}_n$. Since $\nu \in C$ permute the primes above $p$, the
projection $\iota_{\nu}(\nu \pi'_n)$ is a uniformizor for $\K_{n, \nu
  \wp_n}$.
  
As a consequence, we obtain for $x \in \eu{K}_n$, unique $t \in \Z[ C ]$ 
and $u \in U(\K_n)$ such that we have the following
\textit{valuation decomposition}:
\begin{eqnarray}
\label{valdeco}
x = (\pi'_n)^t \cdot u.
\end{eqnarray}

We define 
\[ \eu{K}'_n = \left\{ \ x = (x_{\nu})_{\nu \in C} \ : \forall \nu,
  \hbox{ $v_p(x_{\nu}) = 0$ or $| v_p(x_{\nu}) | \geq r(\K)
    v_p(\underline{\pi}_n)$} \ \right\} \subset \eu{K}_n . \]
Accordingly, consider $s_n := \varsigma^{r(\K)/e(\K)}$ and $S_n =
(s_{\nu})_{\nu \in C}, s_{\nu} = s_n^{\delta_{\nu,1}}$, so $s_n, S_n$
build norm coherent sequences and $S_n \in \eu{K}'_n$. If $\upsilon =
(u_n)_{n \in \N}$ with $u_n \in U(\K_n)$ is any norm coherent
sequence, we say that $\Upsilon = \upsilon S = (u_n S_n)_{n \in \N}$
is a sequence of \textit{pseudouniformizors} for $\eu{K}'_n$.

\begin{definition}
\label{unif} 
With respect to $\Upsilon$ elements of $\eu{K}'_n$ have a
decomposition which is reminiscent of \rf{valdeco}. Let $x = (x_{\nu})
\in \eu{K}'_n$. Then there is a $t \in \Z_p[ C ]$ and a $u(x) \in
U(\K_n)$ with $x = \Upsilon_n^t \cdot u(x)$. We define the maps
$\psi_{\Upsilon,n} : \eu{K}_n \ra U(\K_n)$ by
\begin{eqnarray}
\label{psidef}
 \psi_{\Upsilon,n} (x) = u(x^{r(\K)}), \quad x = \Upsilon_n^t \cdot u(x),
\end{eqnarray}
according to the decomposition  defined above. 

Note that the converse of \rf{valdeco} does not hold in general, since
$\pi^t$ is not defined for arbitrary $p$-adic coefficients. Therefore, we
extend $\eu{K}$ by tensoring with $\Z_p$. Using the uniformizors
$\underline{\pi}_n$ defined in \S \ref{idels}, we have
\[ \eu{K}_n^{\times} = \prod_{\nu \in C} U(\K_{n,\wp_n}) \cdot
\nu(\underline{\pi}_n)^{\Z}, \] and we let 
\[ \wh{\eu{K}}_n = \eu{K}_n^{\times} \otimes_{\Z} \Z_p = \prod_{\nu
  \in C} U(\K_{n,\nu \wp_n}) \cdot \nu(\underline{\pi}_n)^{\Z_p}, \] a
completion which does not depend upon the choice of
$\underline{\pi}_n$.

The submodule $\wh{\eu{K}}'_n$ is defined accordingly.
\end{definition}

\subsection{Fundamental maps}
Note that by raising $x$ to the power $r(\K)$ before applying its
decomposition, we obtain a map $\psi_{\Upsilon,n}$ that is defined on
all $\eu{\K}_n$ and, by definition, $\psi_{n,\Upsilon}$ acts on
$U(\K_n)$ by $y \mapsto y^{r(\K)}$; moreover, if $x \in
(\eu{K}_n)^{p^m}$, then $\psi_{n,\Upsilon}(x) \in U(\K_n)^{r(\K)
  p^m}$.  It can be verified from the definition that $\psi_{n,
  \Upsilon}$ are homomorphisms of $\Z_p$-modules; they are not
homomorphisms of $\Z_p[ \Delta ]$-modules, since $\Delta$ may act on
$\Upsilon_n$. However, we have the following useful fact:
\begin{lemma}
  \label{euc}
  Let $\psi = \psi_{n,\Upsilon}$ be the map defined above; then
  $\psi(\nu \rho_0) = \nu (\psi(\rho_0))$ for all $\nu \in C^+$. For
  all $m > n \geq 1$ and all $x \in \wh{\eu{K}_m}$ we have
\begin{eqnarray}
\label{normcoh}
N_{m,n}(\psi_{m,\Upsilon}(x)) = \psi_{n,\Upsilon}(N_{m,n}(x)). 
\end{eqnarray}
Moreover, for arbitrary $\theta \in \eu{C}$ and arbitrary $n$ we have
\begin{eqnarray}
\label{equic}
\psi_{n,\Upsilon}\left( \rho^{\theta}\right) & = & \psi_{n,\Upsilon}(\rho)^{\theta}, \quad \hbox{and} \\
x \in \K_n^{p^m}, \ y \in U(\K_n) \ & \Rightarrow & \ 
\psi_{n,\Upsilon}(x) \in U(\K_n)^{r(\K) p^m}, \ \psi_{n,\Upsilon}(y) = y^{r(\K)}.\nonumber
\end{eqnarray}
\end{lemma}
\begin{proof}
  By definition of $\rho$ we have $v_{\wp_n}(\rho) = q p^{n-k}$ and
  $v_{\overline{\wp}_n}(\rho) = - q p^{n-k}$ while $v_{\nu \wp_n}(\rho) = 0$ 
for all $\nu \in C \setminus \{1, \jmath\}$. Let
  $\Upsilon_n = (\Upsilon_{\nu})_{\nu \in C} \in \eu{M}(\eu{K}_n)$ be the
  pseudouniformizor with respect to which $\psi_{n,\Upsilon}$ is
  defined. By comparing valuations, we see that there is a unit $u \in
  U(\K_n)$ with $\rho^{r(\K)} =
  (\Upsilon_n/\overline{\Upsilon}_n)^{(r(\K)/e(\K)) q p^{n-k}} \cdot
  u$ and the definition of $\psi_{n,\Upsilon}$ yields
  $\psi_{n,\Upsilon}(\rho) = u$. For arbitrary $t = \sum_{\nu \in C^+}
  c_{\nu} \nu \in \eu{C}$ we have
  \[ \rho^{t} = \prod_{\nu \in C^+} \nu(\Upsilon_n/\overline{\Upsilon}_n)^{q
    p^{n-k} c_{\nu}} \cdot \nu(u)^{c_{\nu}} \in \wh{\eu{K_n}}.  \] The
  definition of $\psi_{n,\Upsilon}$ yields $\psi_{n,\Upsilon}\left(
    \rho^t\right) = u^t = \psi_{n,\Upsilon}(\rho)^t$, which is the first
  claim in \rf{equic}. The second claim is a direct verification.

  For \rf{normcoh}, let $x^{r(\K)} = \Upsilon_m^t \cdot u_x, \ t \in
  \eu{C}, u_x \in U(\K_m)$.  Since the uniformizors are norm coherent,
  we have $N_{m,n}(x)^{r(\K)} = \Upsilon_n^t N_{m,n}(u_x)$ with
  $N_{m,n}(u_x) \in U(\K_n)$. By applying $\psi$, we find
\[ N_{m,n}(\psi_{m,\Upsilon}(x)) = N_{m,n}(u_x) = \psi_{n,\Upsilon}(N_{m,n}(u_x)).\]
This completes the proof.
\end{proof}
The use of $\wh{\eu{K}_n}$ is temporary and it is introduced for
obtaining a $\Z_p$-homomorphism. We will derive later maps which are
defined on modules endowed with their own $\Z_p$-module structure,
thus making the use of $\wh{\eu{K}}$ superfluous. Since the
$(\pi_n)_{n \in \N}$ build a norm coherent sequence, if follows from
Lemma \ref{euc} that the maps $\psi_{n,\Upsilon}$ form themselves a
norm coherent sequence.

Let $\Upsilon = (\pi_n)_{n \in \N}$ be an arbitrary pseudouniformizor
sequence and $\psi_{n, \Upsilon} : \wh{\eu{K}_n} \ra U(\K_n)$ be the
sequence of $\Z_p$-homomorphisms defined in \rf{psidef}.  Our next
purpose is to relate $\psi$ to the module $B'$, in order to define the
map $\wh{\Theta}$ mentioned above. Let $a' \in {A'}^-[ T ]$ be a class
represented by $a = (a_n)_{n \in \N} \in A^-$ and $\theta = \Theta(a
B)$ be such that $\wh{T}(a') = b^{\theta} = a^T$.

 We fix $n > 1$ a large integer and let $m > n$. Let $\eu{A}_{m} \in
 a_{m}$ be a prime ideal coprime to $p$, which is totally split in
 $\K_{m}$ and verifies
\[ \eu{A}_{m}^{(1-\jmath) q p^{m-k}} = (\alpha_{m }), \quad
\alpha_{m} \in \K_{m}^{1-\jmath}. \] At finite levels, using \rf{bn},
we find that there are $\gamma_{m} \in \K_{m}^{1-\jmath}$ and
$t_{m} \in \Z[ C ] \subset \eu{C}$, which are approximants of
$\theta$ to the power $q p^{m-k}$, such that
\begin{eqnarray}
\label{ideals}
\eu{A}_{m}^{(1-\jmath) T} = (\gamma_{m}) \cdot \eu{R}_{m}^{t_{m}}.
\end{eqnarray}

Raising this relation to the power $q p^{m-k}$, and using \rf{rho}, we
obtain $ (\alpha_{m}^T) = (\gamma_{m}^{q p^{m-k}}) \cdot
(\rho^{t_{m}})$. Thus $\alpha_{m}^T \cdot \gamma_{m}^{-q p^{m-k}} =
\xi \rho^{t_{m}}$ for some root of unity $\xi$; taking norms to $\K_1$
on both sides yields $\gamma_1^{-q p^{m-k}} = N(\xi) \cdot \rho^{t_m
  p^{m-k}}$ and thus $N(\xi) \in \K^{p^{m-k}} \cap \mu_{p^k} = \{ 1
\}$. In particular
\[ \xi \in N_{m,1}^{-1}(1) \cap \mu_{p^{m}} \subset \mu_{p^{m}}^T. \]
An adequate choice of $\alpha_{m}$ allows the assumption that $\xi =
1$. We thus obtain the fundamental identities:
\begin {eqnarray}
\label{val}
\alpha_{m}^T \cdot \gamma_{m}^{-q p^{m-k}} & = & \rho^{t_{m}}, \\
\alpha_{n}^T \cdot \gamma_{n}^{-q p^{n-k}} & = & \rho^{t_{m}}.\nonumber
\end {eqnarray}
The lower identity is obtained from the first one by taking the norm
$N_{m, n}$ and then extracting the \nth{p^{m-n}} root. Here
$\alpha_n^{p^n} = N_{m,n}(\alpha_{m})$ so that $(\alpha_n) =
\eu{A}_n^{(1-\jmath) q p^{n-k}}$ with $\eu{A}_n = N_{m,n} \eu{A}_{m}$,
etc. The value $\alpha_n = (N_{m,n}\alpha_{m})^{1/p^n}$ is determined
only up to roots of unity, and it will be chosen such that the two
sides of the equation agree.

Taking in addition the norm $N_{n,1}$ we see that $\gamma_1^{-q
  p^{m-k}} = \rho^{t_{m} p^{m-k}}$ and after taking roots we have
$\gamma_1^{-q} = \zeta^c \rho^{t_{m}}$. It follows that $\rho^{t_{m}}
\in N_{n,1}(\K_{n}^{\times})$ for all $m$. Consequently,
$\rho^{\theta} \in \wh{\id{N}} = \cap_m
N_{m,1}(\wh{\K_{m}}^{\times})$. Indeed, fixing $n$ and letting $m \ra
\infty$ we see that $\rho^{\theta} \in N_{n,1}(\K_{n}^{\times})$. This
holds for all $n$, which confirms the claim:
\begin{eqnarray}
 \label{univnorm}
\rho^{\theta} \in \bigcap_n \wh{N_{n,1}(\K_{m}^{\times})} =: \widehat{\id{N}}.
\end{eqnarray}
 
We have shown in Lemma \ref{deco} and the comments following its
proof, that $x \in (\eu{\K}^{\times})^v$ decomposes uniquely in
$x^{\top} \in \id{N}$ and $x^{\bot} \in \eu{Z}$; the fact still holds
upon tensoring with $\Z_p$. By raising to the power $r(\K)$ in the
definition of $\psi$ we obtain decompositions $\psi_{1,\Upsilon}(\rho)
= u^{\bot} \cdot u^{\top}$ with $u^{\bot} \in \eu{Z}_0$ and $u^{\top}
\in \wh{\id{N}}$.  We may choose a norm coherent sequence of units
$u_n \in U(\K_n)$ such that $u_1^{1-\jmath} = u^{\top}$.  By letting
$\tilde{\pi}_n = \Upsilon_n u_n$ we obtain a new norm coherent
sequence of uniformizors $\tilde{\Upsilon}$. By definition,
\[ \rho^{r(\K)} = (u_1 \Upsilon_1)^{1-\jmath} u^{\bot} =
\tilde{\Upsilon}_1/\overline{\tilde{\Upsilon}_1} \cdot u^{\bot} . \]
Raising to $\theta$ and using the fact that $\eu{Z}$ is
$C^+$-invariant, as established at the end of the previous chapter, it
follows that
\[ \psi_{1,\tilde{\Upsilon}}(\rho^{\theta}) = \psi_{1,\tilde{\Upsilon}}(\rho)^{\theta} = (u^{\bot})^{\theta} 
\in \wh{\id{N}} \cap \eu{Z}_0 = \{ 1 \}. 
\]
It is obvious that $\tilde{\Upsilon}$ is defined up to norm coherent
sequences with $u_1 = 1$.  We have thus proved:
\begin{lemma}
\label{rho1}
We can choose pseudouniformizors $\Upsilon$ such that
$\psi_{1,\Upsilon_1}(\rho^{\Theta(a')}) = 1$ for $a' \in {A'}^-[ T ]$;
the sequences verifying this condition differ by coherent sequences of
units $\upsilon = (u_n)_{n \in \N}$ with $u_1 = 1$.
\end{lemma}

We fix $\psi_n = \psi_{n, \Upsilon}$, a family of maps verifying
$\psi_1(\rho)^{\theta} = 1$ and $N_{j,l}(\psi_j(x)) =
\psi_l(N_{j,l}(x))$ for all $x \in \wh{\eu{K}_j}, j > l > 1$. Then we
have shown that $\psi_1(\rho)^{\Theta(a')} = 1$ for any $a' \in
{A'}^-[ T ]$ and thus, for $\theta_n \in \Z[ C ]$ approximating
$\Theta(a')$ by $\theta_n - \Theta(a) \in p^n \Z_p[ C ]$, we have
  \begin{eqnarray}
  \label{pn}
  \rho^{r(\K) \theta_n} \in (\K^{\times})^{p^n} \cdot  \prod_{\nu \in C} p^{\Z}. 
  \end{eqnarray}
  Since $w | r(\K)$, it follows that $\rho^{r(\K)}$ does not depend on
  the choice of $\rho \in \eu{R}^q$.  If $a'$ generates $ {A'}^-[ T ]$
  as a cyclic module, we have seen that $\wh{T}(a') \in B^- \setminus
  A^p$ and thus $\Theta(a') \not \in p \Z_p[ C ]$. We can a posteriori
  choose $r(\K)$ to be the smallest positive integer such that \rf{pn}
  holds. This explains the map $\wh{\Theta}$ defined in \rf{eval}: we
  write for future reference:
\begin{proposition}
\label{mred}
Let $\theta = \wh{\Theta}(a') \in \Z_p[ C ]$ for some $a' \in {A'}^-[
T ]$.  There is a positive integer $r(\K)$ such that for all
approximants $\theta_n \in \Z[ C ]$ of $\theta$ with $\theta -
\theta_n \in p^n \Z_p[ C ]$ the condition \rf{pn} holds.

If $\KL, \KL_n$ are the fields defined in Lemma \ref{locuni}, then
\[ \F_n := \KL_n[ \rho^{r(\K) \theta_n/p^n} ], \quad \F = \KL[
\rho^{\theta/p^{\infty} } ] := \cup_n \F_n , \] are unramified Kummer
extensions of their base fields and they are galois over $\K$.
\end{proposition}
\begin{proof}
  The first statement was proved in \rf{pn}, we now show that
  $\F_n/\KL_n$ are unramified.  Let $\eu{p}$ be any prime above $p$ in
  $\K_{\infty}$, let $\eu{P} \subset \KL$ be the ramified prime above
  $\eu{p}$. Since the conjugates of $\wp$ are totally ramified in
  $\KL/\K$, we can assume without loss of generality that $\eu{P} \cap
  \K = \wp$. Let $\rg{L} = \KL_{\eu{P}}$ be the completion at
  $\eu{P}$. For every $n > 0$, it follows from \rf{pn} that
  $\iota_1(\rho^{r(\K) \theta_n}) = p^z \cdot \gamma^{p^n}$ for some
  $z \in \Z$ and $\gamma \in \K_{\wp}$. Since $p^{1/p^n} \in \rg{L}$,
  it follows that the extension $\KL[ \rho^{r(\K) \theta_n/p^n} ]$,
  which is unramified outside $p$, is the trivial extension at each
  prime $\eu{P} \supset (p)$, and consequently $\F_n \cdot \KL$ are
  unramified for all $n$, so $\F/\KL$ is unramified. A similar
  argument shows that $\F_n/\KL_n$ is also unramified. The fact that
  $\F/\K$ is galois follows by Kummer theory. First note that $\KL/\K$
  is a metabelian galois extension. The fact that $\F/\K$ is galois
  follows for instance by investigating the action of $\Gal(\KL/\K)$
  on the radical of $\F$.
\end{proof}
The Proposition \ref{mred} confirms herewith also the existence of the
radical producing maps announced in \rf{eval} in the introduction of
this chapter.
\begin{remark}
\label{obstruct}
Note that $\Omega_{E'}(\K_{\infty})/\Omega_{E}(\K_{\infty})$ is a
$p$-ramified extension with galois group isomorphic to $\Z_p^s$. The
statement of the second claim in the above Proposition is an
equivalence, and it implies that the maximal unramified extension
$\overline{\F}/\Omega_{E}(\K_{\infty})$ contained in
$\Omega_{E'}(\K_{\infty})$ has group of essential $p$-rank equal to
$\zprk((A')^-[ T ])$. For this, it suffices to take
\[ \overline{\F} = \Omega_E \cdot \KL\left[ \wh{\Theta}({A'}^-[ T
  ])\pinf\right] \ \subset \ \Omega_{E'}. \] The Gross - Kuz'min
Conjecture herewith states that $\Omega_{E'}(\K)$ is totally ramified
-- up to finite subextensions -- over $\Omega_{E}(\K_{\infty})$. The
field $\overline{\F}$ is an explicit obstruction to the Gross - Kuz'min
Conjecture being true.
\end{remark}

\section{Proof of the Theorem \ref{gc}}
In this chapter we apply the radicals $\wh{\Theta}({A'}^-[ T ])$ and the
insights gained in Proposition \ref{mred},
obtaining large unramified extensions of $\KL_n$. These give
raise to classes $c_n \in A(\KL_n)$ of large order and annihilated by
several augmentation elements: in fact, we obtain a representation of
${A'}^-[ T ]$ in a module $\id{C}$ of sequences of such classes, as
shown in \rf{reps}. Finally, the crucial Proposition \ref{iter} raises by its
repeated application a contradiction to the finiteness of $A^+[ T^* ]$, 
which confirms the Conjecture of Gross and Kuz'min for CM fields.
 
Let $\sigma \in \Gal(\KL/\K_{\infty})$ be a topological generator of
this $\Z_p$-group and let $\tilde{\tau} \in \Gal(\KL/\K)$ be some
lift, with $\tilde{T} = \tilde{\tau}-1$; we write $g = \sigma - 1\in
\Z_p[ \gamma ]$, the generator of the augmentation of this group ring.
Finally, we denote by $\tilde{\jmath} \in \Gal(\KL/\K^+)$ any lift of
complex conjugation.

Let $\KH_n \supset \KL_n$ be the maximal $p$-abelian unramified
extensions of these finite extensions of $\Q$ and let $\KH = \bigcup_n
\KH_n$ be the maximal $p$-abelian unramified extension of $\KL$. 

The $p$-parts of the class groups of $\KL_n$ are $A(\KL_n)$ and they
map to $X_n$ isomorphically via the Artin symbol $\varphi_n$. The 
groups $X_n$ form a projective sequence with respect to the galois
action by restriction: $\rest : X_{n+1} \ra X_n$. We obtain maps
\[ f_{n+1, n} = \varphi_n \circ \rest \circ \varphi_{n+1}^{-1} : A_{n+1} \ra A_n , \]
which arange the classes $A_n$ in a projective sequence, with respect to which
we let $A(\LK) = \varprojlim_n A(\KL_n)$.

We consider the following fields:
\begin{eqnarray*}
\M & = & \KL\left[ \rho^{r(\K) \Theta({A'}^-[ T ])/p^{\infty}} \right] \subset \KH, \\
\M_n & = & \KH_n \cap \M, \\
\underline{\M}_n & = & \KL_n[ \rho^{r(\K) \Theta_n({A'}^-[ T ])_n/p^{n}} ],
\end{eqnarray*}
with the exponent maps $\Theta, \Theta_n$ defined in \rf{thetas}.  Let
$X(\KL) = \Gal(\KH/\KL)$ and $X^{\circ}(\KL) \subset X(\KL)$ be its
$\Z_p$-torsion.  If $\overline{\KH} = \KH^{X^{\circ}}$ is the fixed
field of the torsion, then $\M \subset \overline{\KH}$, since its
radical and thus its galois group are $\Z_p$-free.  We obviously have
$\underline{\M}_n \subseteq \M_n$ while $\bigcup_n \underline{\M}_n =
\bigcup_n \M_n = \M$: but can the two fields differ for some $n$?
Thus, is there an $n$ such that $\underline{\M}_n \neq \M_n$? Before
answering this question, we need some information about galois
actions.
\begin{lemma}
\label{compl}
In the notation introduced above, there is a subfield $\M' \subset
\KH$ such that $\KH = \M \cdot \M'$ and $\left[ \M \cap \M' : \KL
\right ] =: D < \infty$.  Let $M = \Gal(\M/\KL)$ and $M_n =
\Gal(\M_n/\KL_n); \ \underline{M}_n =
\Gal(\underline{\M}_n/\KL_n)$. Then $g, \tilde{T}$ and
$\tilde{\jmath}-1$ annihilate the galois groups $M$ and $M_n$.
\end{lemma}
\begin{proof}
  Let $\D \subset \overline{\KH}$ be cyclic; since $X(\KL)$ is an
  abelian $p$-group, it follows that the infinite cyclic subgroups are
  isomorphic to $\Z_p$, so $\D/\KL$ is a $\Z_p$-extension; since
  $\Z_p$ has no infinite subgroups, infinite galois theory implies
  that $\D \cap \M$ is finite or $\D \subset \M$. Since
  $\Gal(\overline{H}/\KL)$ has an at most countable set of generators,
  we can find by induction some extension $\M'' \subset
  \overline{\KH}$ with $\M'' \cap \M$ of finite exponent over $\KL$
  and $\M'' \cdot \M = \overline{\KH}$. Since $\Gal(\M/\KL)$ has
  finite $p$-rank, it follows that in fact $\M'' \cap \M$ is finite
  over $\KL$. Finally, $\Gal(\KH/\overline{\KH})$ is a $\Z_p$-torsion
  group, and using again the fact that $\Gal(\M/\KL)$ is finitely
  generated, we conclude that there is some extension $\M' \subset
  \KH$ such that $\M' \cap \overline{\KH} = \M''$ and the statement of
  the lemma holds for $\M'$. The annihilation of $\underline{M}_n$ by
  $g, \tilde{T}, \tilde{\jmath}-1$ is an application of Kummer pairing
  using the fact that $T, g, \jmath+1$ annihilate the Kummer radical
  $\rad(\underline{M}_n/\KL_n)$. Since $\M = \bigcup_n
  \underline{\M}_n$, the groups $\underline{M}_n$ form a projective
  sequence with respect to restriction and it follows that $M$ is also
  annihilated by the three augmentation elements. We can now restrict
  $M$ to $M_n$ and deduce that the same holds for the groups $M_n$,
  which completes the proof.
\end{proof}

As a consequence:
\begin{lemma}
\label{lsexp}
Let $v_p(r(\K)) = b \geq 0$. Then for all sufficiently large $n$ we
have $\M_n = \underline{\M}_n$ and
\[ \sexp(M_n) = p^{n-b} = \exp(M_n) . \]
\end{lemma}
\begin{proof}
  Let $n > 0$ be sufficiently large and let $\T' \subset \M_n$ be a
  maximal cyclic extension with intersection $\T = \T' \cap
  \underline{\M}_n$. If $\T = \T'$ for all $\T$, there we have $\M_n =
  \underline{\M}_n$. We assume this is false, so we can choose $\T'$
  such that $\T \subsetneq \T'$. Since $\T' \subset \M = \bigcup_n
  \underline{\M}_n$, there is some $m \geq n$ such that $\T' \cdot
  \KL_m \subset \underline{\M}_m$; we assume that $m$ is minimal with
  this property and let $a'_m \in \in ({A'}^-[ T ])_m$ be chosen such
  that for an integral approximants $t_m \in \Z[ C ]$ with $t_m -
  \Theta_m(a')_m \in p^m \eu{C}$ and $t_m - \Theta_m(a')_m \in p^m
  \eu{C}$ we have $\T' \cdot \KL_m = \KL_m[ \rho^{ r(\K) t_m/p^m} ]$;
  since $\T \in \underline{\M}_n$, it follows that $\T = \KL_n[ \rho^{
    r(\K) t_n/p^n} ]$ with $t_n - t_m \in p^n \Z[ C ]$. Moreover, the
  choice of $m, t_m$ imply that $a'_m \not \in (A^-_m)^p$.

  Suppose first that $b > 0$; the definition of $r(\K)$ implies that
  the field extension $\KL_n[ \rho^{(t \cdot r(\K)/p )/p^n} ]/\KL_n$
  is $p$-ramified for any $t \in \Z[ C ]$, in particular also for
  approximants of some $\Theta_m(a'_m)$. In particular $\T'$ must be
  ramified above $\T$, which settles the case $b > 0$. Suppose now
  that $b = 0$.  Then $\KL_m \T'$ should be abelian over $\KL_n$,
  since $\T' \cap \KL_m = \KL_n$, as one verifies by comparing
  ramification of the two extensions of $\KL_n$. Note that since $b =
  0$, we have $[ \T : \KL_n ] = p^n$. A common galois theoretic
  argument on Kummer radicals shows that in this case $\K_m \T' =
  \K_m[ \rho^{ r(\K) t_m/p^m} ]$ cannot be abelian over $\KL_n$: for
  obtaining abelian Kummer extension, the radical should be
  annihilated by $\omega_n^* \cong p^n \bmod (T, p^{n+1})$. A fortiori
  $\KL_m \T'/\KL_m$ cannot be abelian unless $\T = \T'$. Since this
  holds for all maximal cyclic subextensions $\T' \subset \M_n$, it
  follows that $\underline{\M}_n = \M_n$. From the definition of
  $r(\K)$ and \rf{pn}, we see that $\sexp(M_n) = \exp(M_n) = p^{n-b}$,
  which completes the proof.
\end{proof}
Let $\D = \M \cap \M'$, let $p^d = \exp(\D)$ and let $n_0$ be such
that $[\M_n \cap \D : \KL_n ] = [ \D : \KL ]$ for all $n \geq n_0$ and
let $\D' = \M_{n_0} \cap \D$.  Then $\M_n \cap \D = \KL_n \cdot \D'$
for all $n \geq n_0$.  Let $D_n \subset A_n(\KL_n)$ be the subgroups
defined by $\M' \cap \KH_n = \KH_n^{\varphi(D_n)}$.

Let $a' \in {A'}^-[ T ]$ be a fixed generator with $\Theta(a') =
\theta$.  Let $1 = g_1, g_2, \ldots g_r \in C$, with $r = \prk({A'}^-[
T ])$, be a set of automorphisms such that $g_i a'$ form generate
${A'}^-[ T ]$ as a free $\Z_p$-module. For $x \in {A'}^-[ T ]$
we write
\[ \F_x = \bigcup_n \ \KL_n[ \rho^{r(\K) \Theta_n(x)/p^n} ] =
\KL[ \rho^{\Theta(x)/p^{\infty}} ] \] In particular, letting
$\F'_i = \F_{g_i a'} \subset \M$ we see that $\F_i$ are pairwise
disjoint and span $\M$, by definition of their radicals. We let
\[ \overline{\F}'_i = \M' \cdot \prod_{j \neq i; j = 1}^r \F'_j . \]
For every $n > n_0$ there is a uniquely defined $c'_{i,n} \in
A(\KL_n)$ such that $\varphi(c'_{i,n}) \big \vert_{\overline{\F}'_i} =
0$ and $\Gal(\F'_i/(\D \cap \F'_i)) = < \varphi(c'_{i,n}) >$. 
Let the sequences $c^{'(i)} = (c'_{i,n})_n \in A(\KL)$ we the projective limits 
defined above; they have infinite order and are annihilated by $\tilde{T}$.  
We thus obtained maps
\begin{eqnarray}
\label{reps}
\id{L} & \ : \ & {A'}^-[ T ] \ra A(\KL)[  \ \tilde{T}^*, g; 1-\tilde{\jmath} \ ], \\ \nonumber
\id{L}_n & \ : \ & \ ({A'}^-[ T ])_n \ra A(\KL_n)[  \ \tilde{T}^*, g; 1-\tilde{\jmath} \ ], \quad n > n_0 .
\end{eqnarray}
Note $\ord(\id{L}(x)) = \infty$ for all $x \in {A'}^-[ T ]$, but
$\ord(\id{L}_n(x))$ may not be maximal (i.e. $p^{n-b}$), the
obstruction being the intersection field $\D$. We let $\id{C} =
\id{L}({A'}^-[ T ])$ and $\id{C}_n = \id{L}_n({A'}^-[ T ])$.
There is a fixed constant $b' \geq b$ such that $\sexp(\id{C}_n) =
p^{n-b'}$ for all $n \geq n_0$ and $p^{b-b'} = \exp(\D)$.

It is natural to consider the canonic extension $\Omega_{E'}(\KL_n)[
c(a')^{1/p^n} ]$, which is a $p$-ramified extension. It will be useful
to define $\K'_n = \K_n[ p^{1/p^{b'}} ]$, so $[ \KL_n : \K'_n ] = p^{n
  - b'}$. We choose $c = (c_n) \in \id{C}$ with $\ord(c_n) = p^{n -
  b'}$ and let explicitly $\U_n = \KL_n[ \beta_n^{1/\ord(c_n)} ]$ for
some $\beta_n \in \eu{B}^{\ord(\eu{B})}$ and $\eu{B} \in c_n$, so that
$\U_n \Omega_E(\KL_n)$ is a canonical extension. In general, one
cannot say more about $\beta_n^g$ except that $\beta_n^g \in
\id{O}^{\times}(\KL_n)$. The next Lemma can be applied iteratively and
leads to a contradiction to the finiteness of $A^+[ T^* ]$


\begin{proposition}
\label{iter}
Consider a tower of extensions $\K_n \subseteq \K' \subset \KL'
\subseteq \KL_n$ and let $\sigma' \in \Gal(\KL'/\K')$ be a generator,
$g' = \sigma'-1$ and $\tilde{T}, \tilde{\jmath}$ act by restriction on
$\KL'$. Suppose that $[ \KL' : \K' ] = p^v = p^{n-u}, u \geq 0$ and
consider $c \in A(\KL'_n)$, a class verifying:
\begin{itemize}
\item[ 1. ] For $a \in A := \{ g', \tilde{\jmath} -1, \tilde{T}^* \}$
  we have $c^a = \{ 1 \}$.
\item[ 2. ] The Norm $\Norm_{\KL'/\K'}(c) = 1$.
\item[ 3. ] The order $q = \ord(c) \geq p^v$.
\end{itemize}
Then $q = p^v$ and there is a class $d \in A(\K')$ which lifts to $c$:
\[  c = \iota_{\KL'/\K'} ( d ), \quad d \in A(\K'). \]
\end{proposition}
\begin{proof}
  We shall break down the proof of the Proposition in three steps,
  starting with two technical Lemmata. The notations and premises will
  be the same as in the hypothesis of the Proposition. We let $N =
  N_{\KL'/\K'} = p^{n-u} + g' F_1(g')$ for some distinguished
  polynomial $F_1 \in \Z[ X ]$.

\begin{lemma}
\label{liter1}
Under the hypotheses of the Proposition, $q = p^v$ and for any ideal
$\eu{c} \in c$, there are $\beta, \beta_g \in \KL'$ and $e_n \in \K'$
such that
\begin{eqnarray*}
\begin{array}{l c l c l c l }
e_n & = & N(\beta_g), & \quad \quad & (\beta_g) & = & \eu{c}^{g'}, \\
(\beta) & = & \eu{c}^q & \quad \quad & \beta^{g'} & = & e_n \beta_g^{q} .
\end{array}
\end{eqnarray*}
\end{lemma}
\begin{proof}
Let $\eu{c} \in c$ be any ideal and $\gamma \in \eu{c}$.  Since
$c$ is annihilated by $a \in A$, we have
 \begin{eqnarray}
 \label{A}
  \gamma^{g'}, \gamma^{\tilde{T}^*}, \gamma^{\tilde{\jmath}-1} 
  \in ({\KL'}^{\times})^{q} \cdot E'(\KL'). 
 \end{eqnarray} 
 Let $\eu{c}^{g'} = (\beta_g), \beta_g \in \KL'$. As discussed above,
 $c$ contains no ramified primes, so in particular $\eu{c}$ is not
 ramified. Therefore, $\beta_g \not \in E(\KL') \cdot
 (\KL_n^{\times})^{g'}$ and we deduce $(\gamma^{g'}) = \eu{c}^{g' q} =
 (\beta_g^q)$. There is a unit $\varepsilon_g \in E(\KL')$ with
 $\gamma^{g'} = \varepsilon_g \beta_g^q$.  We have $(N(\gamma_g)) =
 N(\eu{c}^{g'}) = (1)$ and there is thus a further unit $e_n \in \K'$
 such that $N(\beta_g) = e_n$. By Hypothesis 2. of the Proposition,
 there is an $\alpha_0 \in \K'_n$ with $(\alpha_0) = N (\eu{c})$, so
 \[ \alpha_0 \id{O}(\KL_n) = \eu{c}^{p^v + g' F(g')} = \eu{c}^{p^v}
 \cdot \beta_g^{F(g')}), \] and there is an $e \in E(\KL_n)$ with $e
 \alpha^{q p^{-v}}_0 = \gamma \cdot \beta_g^{q p^{-v}F_1(g')}$.  It
 follows for $\beta = \gamma/e \in \KL'$ that $\beta =
 \left(\alpha/\beta_g^{F_1(g')}\right)^{q p^{-v}}$.  Let $x =
 \alpha/\beta_g^{F_1(g')} \in \KL'$ and $q' = q p^{-v} \geq 1$. We
 obtained
 \[ ( \beta ) = \eu{c}^{q} = ( x^{q'} ) \quad \Rightarrow \quad \left(
   \eu{c}^{q/q'} /(x)\right)^{q'} = (1). \] Since $\ord(\eu{c}) =
 \ord(c) = q$, we must have $q' = 1$ and $q = p^v$, which is the first
 claim. With $q = p^v = [ \KL' : \K' ]$ we proceed by acting with $g'$
 upon the previous identity, in which we introduced $\beta_g^{g'
   F_1(g')} = \beta_g^{N-p^v}$, thus obtaining:
\begin{eqnarray}
\label{ats}
\quad \quad \beta^{g'} = e_n^{-1} \beta_g^{p^v} \quad \hbox{with} 
\quad (\beta_g) = \eu{c}^g \quad \hbox{and} \quad 
e_n = N(\beta_g) \in \K'_n.
\end{eqnarray}
This completes the proof of the Lemma, with $\beta, \beta_g, e_n$.
\end{proof}

The next result concerns abelian Kummer extensions using the radical $\beta$:
\begin{lemma}
\label{liter2}
Let $\eu{c} \in c$; $ \beta, \beta_g \in \KL'$ and $e_n \in E(\K')$ be
like in the Lemma \ref{liter1}. Consider the abelian Kummer extension
$\V' = \KL'[ \beta^{1/q} ]$; then $\V'/\KL'$ is unramified outside $p$
and abelian over $\K'$. There is a Kummer extension $\U' / \K'$ which
is fixed by a lift of $\sigma'$ to $\Gal(\V'/\K')$ and is unramified
outside $p$; moreover
\[ \Omega_E(\KL') \cap \V' = \KL' \quad \hbox{ and } \quad \Omega_{E}(\K') \cap \U' = \K'. \]
\end{lemma}
\begin{proof}
  With $\beta$ like in \rf{ats}, we let $\V' = \KL'[ \beta^{1/q} ]
  \subset \Omega_{E'}(\KL')[ c^{1/q} ]$.  Since $(\beta) =
  \eu{c}^{\ord(\eu{c})}$, the extension $\V'/\KL'$ is unramified
  outside $p$. Moreover, we have $\V' \cap \Omega_{E'}(\KL') = \KL'$:
  if not, the intersection $\Omega_{E'}(\KL') \cap \U' \neq \KL'$, so
  it contains a cyclic extension of degree at least $p$. Kummer theory
  yields then that there is a unit $\delta$ such that $\beta \delta =
  x^p \in ({\KL'}^{\times})^p$. Since $\beta \delta \in (\beta)$, it
  follows that $\eu{c}^{\ord(c_n)} = (x^p)$ for some $x \in \KL'$,
  which is absurd. Therefore $\V' \cap \Omega_{E'}(\KL') = \KL'$.

  Let $\K'' = \K'[ e_n^{1/q} ]$ and $\KL'' = \KL' \cdot \K''$; let
  $\chi \in \Gal(\K''/\K') = \Gal(\KL''/\KL')$ be a generator of this
  cyclic group and let $\V'' = \KL'' \cdot \V'$. Since $\V', \KL''$
  are linearly disjoint Kummer extensions of $\KL'$ while $\KL'' =
  \KL' \cdot \K''$, it follows that $\chi$ acts trivially on
  $\Gal(\V'/\KL')$.

  We have shown that $\sigma'(\beta) e_n = \beta x^q$, for some $x \in
  \KL'$; it follows that $\V''$ is galois over $\K''$. The
  equivariance of the Kummer pairing implies that $\Gal(\V''/\K'')$ is
  invariant under the action of $\sigma'$ by
  conjugation. Consequently, $ \V''/\K''$ is abelian, so there is a
  cyclic extension $\U''/\K''$ which is fixed by a lift
  $\tilde{\sigma} \in \Gal(\V''/\K'')$ of $\sigma'$, and such that
  $\V'' = \U'' \cdot \KL'$. We have seen that $\chi \in
  \Gal(\K''/\K')$ acts trivially on $\Gal(\V''/\KL'') \cong \Gal(
  \U''/\K'')$. If follows that $\U''/\K'$ is abelian and we let $\U'
  \subset \U''$ be fixed by a lift of $\chi$ to $\Gal(\U''/\K')$. By
  construction we have $\U' \cdot \KL' = \V'$; since $\V' \cap
  \Omega_{E'}(\KL') = \KL'$, if follows a fortiori that $\U' \cap
  \Omega_{E'}(\K') = \K'$: indeed, if $\K' \subset \D' =
  \Omega_{E'}(\K') \cap \U' \subset \U'$, then $\D'' := \D' \cdot \KL'
  \subset \Omega_{E'}(\KL') \cap \V' = \KL'$, as we have shown
  before. Since $[ \V' : \KL' ] = [ \U' : \K' ] = q$, we have $\U'
  \cap \KL' = \K'$, and we conclude that $\U' \cap \Omega_{E'}(\K') =
  \K'$, as claimed. Since $\KL'/\K'$ is unramified outside $p$ and so
  is $\V'/\KL'$ it follows that $\U' \subset \V'$ is unramified
  outside $p$, over $\K'$, which completes the proof of the Lemma.
\end{proof}
 
We can now complete the proof of the Proposition \ref{iter}.  Since
the cyclic Kummer extension $\U'/\K'$ is $p$-ramified and $\U' \cap
\Omega_{E'}(\K') = \K'$, it follows from ramification theory that the
radical $\upsilon \in \K'; \ \U' = \K'[ \upsilon^{1/p^v} ]$ is power
of an ideal.  More precisely, there is some class $d \in A(\K')$ with
$\ord(d) = p^v = [ \U' : \K' ]$, together with an ideal $\eu{D} \in d$
so that $(\upsilon) = \eu{D}^{p^v}$.  From $\V' = \KL' \cdot \U'$ we
deduce $\KL'[ \upsilon^{1/p^v} ] = \KL'[ \beta^{1/p^v} ]$ and Kummer
theory shows that there is some integer $(u,p) = 1$ such that $\beta
\cdot \upsilon^u = x^{p^v} \in ({\KL'}^{\times})^{p^v}$.  We can
assume without loss of generality that $u = -1$ so, in terms of
ideals, we conclude that

\begin{eqnarray*}
 (\beta) & = & \eu{c}^{p^v} = (\upsilon x^{p^v}) = ((x) \eu{D})^{p^v} \quad \Rightarrow \quad 
\eu{c} = (x) \cdot \iota_{\KL'/\K'}(\eu{D}),   \\
 { } [ \eu{c} ] & = & c = [  \iota_{\KL'/\K'}(\eu{D}) ] = [ \iota_{\KL'/\K'}(d) ],
\end{eqnarray*}
so $c$ is the lift of a class from $\K'$, and $\ord(d) = \ord(c) = q =
p^v$, which completes the proof.
\end{proof}
We may now complete the proof of Theorem \ref{gc} by a repeated
application of the Proposition \ref{iter}.
\begin{proof}
  Let $q' = p^a = \exp(A^+(\K)( T^* ))$, which is a finite module by
  Lemma \ref{noT*}, and let $n > n_0 + 2( a + b')$. Assuming that
  ${A'}^-(T) \neq \{ 1 \}$, we constructed $\id{C} = \id{L}({A'}^-[ T
  ])$ and we choose $c \in \id{C}$ with order $p^v = p^{n-b'}$. We
  apply the Proposition first to $\KL' = \KL_n, c = c_n$ and $\K' =
  \K'_n$, the degree defining herewith the extension $\K'$. The
  Hypotheses 1. and 3. of the Proposition are satisfied by definition,
  while $N(c) = c^{p^v + g' F_1(g')} = 1$, since $c^{g'} = 1$ and
  $\ord(c) = p^v$.  This shows that Hypothesis 2. is also
  verified. The proposition implies that $c$ is the lift of some ideal
  $d \in A(\K'_n)$.
  
  In the next step we apply the Lemma to the fields $\KL' = \K'_n$ and
  $\K' = \K_n$; we let $c = d^{p^a}$, where $d \in A(\K'_n)$ is the
  class which resulted from the first application of the
  Proposition. From the properties of $d$ it follows that $d' :=
  \Norm_{\K'_n/\K_n}(d) \in A^+( T^* )$ and thus, in this case,
  $N(d^{p^a}) = N(c) = (d')^{p ^a} = 1$, which confirms that the
  Hypothesis 2. of the Lemma applies to our data.  Moreover, $\ord(c)
  = p^{v - a} = p^{n - a - b'} > p^{b'} = [ \K'_n : \K_n ] = [ \KL' :
  \K' ]$. The Proposition implies in this case that we must in fact
  have $\ord(c) = p^{b'}$ and $n = 2 b' + a$, which contradicts our
  choice. Note that the second application of the Proposition
  essentially raises a contradiction to the finiteness of $A^+[ T^*
  ]$, as predicted.  The assumption ${A'}^-[ T ] \neq \{ 1 \}$ is thus
  untenable, which completes the proof.
\end{proof}
We end with a brief review of the crucial steps of the proof of
Theorem \ref{gc}.  We proved that $A^+[ T^* ]$ must always be finite
in CM fields and led this to a contradiction of the assumption that
${A'}^-[ T ] \neq \{ 1 \}$. Based on this assumption, we developed
a series of maps that yield in fact isomorphic images in modules such
as $B^-$, in Lemma \ref{t2}, the local multiplicative group
$\prod_{\nu \in C} p^{\Z}$, in Proposition \ref{mred} and finally in
$A(\KL)$, in \rf{reps} and the Lemmata that lead to that
result. Finally, a twofold application of Proposition \ref{iter} raised a
contradiction to the finiteness of $A(\K)^+[ T^* ]$ which shows that
${A'}^-[ T ]$ must be trivial.
\begin{remark}
  Like in the case of Leopoldt's conjecture, it is obvious that the
  Gross-Kuz'min conjecture follows from the Conjecture of
  Shanuel. Less than this conjecture would in fact suffice, e.g. a
  generalization of Baker's result to homogeneous forms in $p$-adic
  logarithms for arbitrary degrees.
\end{remark}

\textbf{Acknowledgment:} I thank Ralph Greenberg, Vicen\c{t}iu
Pa\c{s}ol, Inder-Bir Passi and Machiel van Frankenhuijser for helpful
discussions and comments during the writing of preliminary versions of
this paper.  I am particularly grateful to Grzegorz Banaszak who
helped by detailed discussions on the choice of uniformizors to find
the proof of the general case presented here. I owe particular thanks
to Jens Franke for his careful critical reading and discussion which
lead to rectifications and improvements in the present version.

\bibliographystyle{abbrv}
\bibliography{gross-ext}

\end{document}